\documentclass[letterchapter,11pt]{article}

\usepackage{latexsym}
\usepackage{amssymb}
\usepackage{amsmath}
\usepackage{graphicx}
\usepackage[english]{babel}

\usepackage{fontenc}
\usepackage{amsfonts}
\usepackage{amsmath}
\usepackage{amsthm}
\usepackage{enumerate}
\usepackage{textcomp}
\usepackage{geometry}
\usepackage{mathrsfs}

\usepackage{fancyhdr}

\usepackage[latin1]{inputenc}

\usepackage{amsfonts}
\usepackage{amsthm}
\usepackage{latexsym}

\newcommand{\RR}{{\rm I}\kern-0.18em{\rm R}}
\newcommand{\h}{{\rm I}\kern-0.18em{\rm H}}
\newcommand{\K}{{\rm I}\kern-0.18em{\rm K}}

\newcommand{\1}{{\rm 1}\kern-0.22em{\rm I}}

\font\tenmath=msbm10 \font\sevenmath=msbm7 \font\fivemath=msbm5
\newfam\mathfam \textfont\mathfam=\tenmath
\scriptfont\mathfam=\sevenmath \scriptscriptfont\mathfam=\fivemath

\font\teneusb=eusb10 \font\seveneusb=eusb7 \font\fiveeusb=eusb5
\newfam\eusbfam \textfont\eusbfam=\teneusb
\scriptfont\eusbfam=\seveneusb \scriptscriptfont\eusbfam=\fiveeusb
\def\eusb{\fam\eusbfam}
\def\frak R{\eusb R}

\font\tenams=msam10 \font\sevenams=msam7 \font\fiveams=msam5
\newfam\amsfam \textfont\amsfam=\tenams
\scriptfont\amsfam=\sevenams \scriptscriptfont\amsfam=\fiveams

\newtheorem{rem}{Remark}
\newtheorem{prop}{Proposition}
\newtheorem{corr}{Corollary}
\newtheorem{lem}{Lemma}

\newtheorem{theo}{Theorem}
\newtheorem{defi}{Definition}

    \font\tenbifull=cmmib10 
    \font\tenbimed=cmmib7
    \font\tenbismall=cmmib5
\mathchardef\bbGamma="7000 \mathchardef\bbDelta="7001
\mathchardef\bbPhi="7002 \mathchardef\bbAlpha="7003
\mathchardef\bbXi="7004 \mathchardef\bbPi="7005
\mathchardef\bbSigma="7006 \mathchardef\bbUpsilon="7007
\mathchardef\bbTheta="7008 \mathchardef\bbPsi="7009
\mathchardef\bbOmega="700A \mathchardef\bbalpha="710B
\mathchardef\bbbeta="710C \mathchardef\bbgamma="710D
\mathchardef\bbdelta="710E \mathchardef\bbepsilon="710F
\mathchardef\bbzeta="7110 \mathchardef\bbeta="7111
\mathchardef\bbtheta="7112 \mathchardef\bbiota="7113
\mathchardef\bbkappa="7114 \mathchardef\bblambda="7115
\mathchardef\bbmu="7116 \mathchardef\bbnu="7117
\mathchardef\bbxi="7118 \mathchardef\bbpi="7119
\mathchardef\bbrho="711A \mathchardef\bbsigma="711B
\mathchardef\bbtau="711C \mathchardef\bbupsilon="711D
\mathchardef\bbphi="711E \mathchardef\bbchi="711F
\mathchardef\bbpsi="7120 \mathchardef\bbomega="7121
\mathchardef\bbvarepsilon="7122 \mathchardef\bbvartheta="7123
\mathchardef\bbvarpi="7124 \mathchardef\bbvarrho="7125
\mathchardef\bbvarsigma="7126 \mathchardef\bbvarphi="7127

\usepackage[colorlinks=true,linkcolor=red,citecolor=blue]{hyperref}
\date{}

\begin{document}

\title{Asymptotic distribution of independent random vectors given their sum}

\author{D. Rabenoro}




\date{Received: date / Accepted: date}

\maketitle

\begin{abstract}
In this paper, we present a conditional principle of Gibbs type for independent nonidentically distributed random vectors. We obtain this result by performing Edgeworth expansions for densities of sums of independent random vectors. 
\end{abstract}


\section{Introduction}
\label{intro}

Let $(X_{i})_{i \geq 1}$ be a sequence of independent not necessarily identically distributed (i.n.n.i.d.) random variables, valued in $\mathbb{R}^{d}$ with $d \geq 1$, except in section \ref{IntroGCP}. Throughout this paper, such a variable is called a random vector (r.v.). We denote by $\mathcal{B}^{d}$ the Borel $\sigma$-algebra of $\mathbb{R}^{d}$ and by $\Lambda^{d}$ the Lebesgue measure on $\mathcal{B}^{d}$. When $\mathcal{E}=\mathbb{R}^{d}$, our main result is a conditional principle of Gibbs type for the $X_{i}$'s, which is closely related to the Gibbs Conditioning Principle (GCP) (see \cite{Stroock and Zeitouni 1991}), recalled below.

\subsection{The Gibbs Conditioning Principle}\label{IntroGCP}
 
Here, we assume that the $X_{i}$'s  are \textit{i.i.d.}, with common distribution $P_{X}$, valued more generally in a Polish space $\mathcal{E}$. Let $f$ be a measurable function from $\mathcal{E}$ to $\mathbb{R}$. Let $a \in \mathbb{R}$ and $\delta >0$. For $1 \leq k<n$, let $R_{nak}^{f,\delta}$ be the distribution of $X_{1}^{k} := (X_{i})_{1 \leq i \leq k}$ given  $\left\{ \frac{1}{n} \sum\limits_{i=1}^{n} f \left( X_{i} \right) \in [a-\delta, a+\delta] \right\}$. Under suitable assumptions on $P_{X}$, the GCP asserts that for $a \neq \mathbb{E}_{P_{X}}(f)$ and any Borel set $B \subset \mathcal{E}^{k}$, 
\begin{equation}\label{gcp}
\lim\limits_{\delta \rightarrow 0} \enskip \lim\limits_{n \rightarrow \infty} R_{nak}^{f,\delta}(B) = (\gamma^{a})^{\otimes k}(B), 
\end{equation}

\noindent
where $\gamma^{a}$ is a so-called Gibbs measure. Namely, $\gamma^{a}$ is defined by 
\begin{equation}\label{gammaa}
\frac{d\gamma^{a}}{dP_{X}}(x) = 
\frac{\exp \left( \theta^{a} f(x) \right)}{\int_{\mathcal{E}} \exp \left( \theta^{a} f(u) \right) dP_{X}(u)}, \enskip \textrm{for } x \in \mathcal{E}, 
\end{equation}

\noindent
where $\theta^{a} \in \mathbb{R}$ satisfies that $\mathbb{E}_{\gamma^{a}}(f)=a$. Further, $\gamma^{a}$ minimizes the relative entropy $H( \cdot | P_{X})$ subject to a suitable constraint. We refer to \cite{Csiszar 1984} for results of that type. Following \cite{Stroock and Zeitouni 1991}, the GCP has a natural interpretation in Statistical Mechanics, where $f$ is seen as the energy. Then, the GCP describes the distribution of a typical small subset in a system composed of a very large number of identical particles, under a condition of fixed averaged energy. The GCP has been extended in several directions. In \cite{Dembo and Zeitouni 1996}, the GCP is established for increasingly large sub-systems, that is, the fixed $k<n$ is replaced by a sequence $(k_{n})_{n \geq 1}$, where $k_{n}=o(n)$ as $n \rightarrow \infty$. Then, in \cite{Leonard and Najim 2002}, an extension of Sanov's Theorem allows to consider more general assumptions for the GCP. Finally, in \cite{Cattiaux and Gozlan 2007}, the enlargement of the conditioning event also depends on $n$, so that there is only one limit in $(\ref{gcp})$. While all these generalizations concern the i.i.d. case and an extension of the GCP has been obtained in the Markov case in \cite{Schroeder 1993}, the independent non identically distributed case has not been investigated yet. The main result of this paper is a Gibbs-type conditional limit theorem for this case.

\subsection{The conditioning point approach}

The proof of the GCP relies on the exchangeability of the $X_{i}$'s under the conditioning event, which holds when the $X_{i}$'s are i.i.d. This allows to interpret this event in terms of the empirical distribution of the $X_{i}$'s and then apply Sanov's large deviations theorem (see \cite{Dembo and Zeitouni 1993}). \textit{However, this exchangeability does not hold anymore in the i.n.n.i.d. case}. In order to overcome this lack, we consider the conditioning point approach, where in the GCP, instead of enlarging the conditioning event as in $(\ref{gcp})$, one considers the distribution $R_{nak}^{f}$ of $X_{1}^{k}$ given $\left\{\sum\limits_{i=1}^{n} f \left( X_{i} \right) = na \right\}$, in the sense of the regular conditional distribution (see \cite{Freedman 1983}) for $X_{1}^{k}$ given $\sum\limits_{i=1}^{n}f\left( X_{i} \right)$. This approach has been used in \cite{Diaconis and Freedman 1988}, where for a sequence $(X_{i})_{i \geq 1}$ of i.i.d. r.v.'s valued in $\mathbb{R}$, the distribution $Q_{nak}$ of $X_{1}^{k}$ given $\left\{\sum\limits_{i=1}^{n} X_{i}=na \right\}$ is studied. Therein, \textit{and in the sequel}, $k$ denotes $k_{n}$, where $(k_{n})_{n \geq 1}$ is a sequence of integers. In \cite{Diaconis and Freedman 1988}, it is assumed that the $X_{i}$'s have a density with respect to (w.r.t.) $\Lambda^{1}$, which implies that $Q_{nak}$ has a density $q_{nak}$ w.r.t. $\Lambda^{k}$. Then, a saddlepoint approximation of $q_{nak}$ (after normalization) provides the following limit theorem, where $d_{TV}$ denotes the total variation distance. 

\begin{theo}(Theorem 1.6. in \cite{Diaconis and Freedman 1988})
\label{DFtheo}
Let $(X_{i})_{i \geq 1}$ be a sequence of i.i.d. r.v.'s valued in $\mathbb{R}$, that is $d=1$. Under suitable regularity assumptions, uniformly in $a$,
\begin{equation}\label{DFconcluTheo}
d_{TV}\left( Q_{nak};\left(\gamma_{I}^{a}\right)^{\otimes k}\right) = 
O\left( \frac{k}{n}\right), 
\end{equation}

\noindent
where $\gamma_{I}^{a}$ is a Gibbs measure on $\mathbb{R}$ defined by 
\begin{equation*}
\frac{d\gamma_{I}^{a}}{dP_{X}}(x) = \frac{\exp \left( \theta^{a}x \right)}{\int_{\mathbb{R}} \exp \left( \theta^{a}u \right)dP_{X}(u)}, \quad x \in \mathbb{R}. 
\end{equation*}

\noindent
Here, $\theta^{a} \in \mathbb{R}$ satisfies that $\mathbb{E}_{\gamma_{I}^{a}}(I)=a$, where  $I$ is the identity function on $\mathbb{R}$.
\end{theo}
   
\noindent
We discovered that this point approach allows to extend the GCP to i.n.n.i.d. $X_{i}$'s. However, the strategy of \cite{Diaconis and Freedman 1988} is not adapted to $R_{nak}^{f}$, which (in general) does not have a density on $\mathcal{E}^{k}$. We rather consider, for $X_{i}$'s valued in $\mathbb{R}^{d}$ and $a \in \mathbb{R}^{d}$, the distribution (still denoted by $Q_{nak}$) of $X_{1}^{k}$ given $\left\{ S_{1,n}=na \right\}$, where $S_{1,n}:=\sum\limits_{i=1}^{n} X_{i}$. Then, by Lemma \ref{densCond} below, $Q_{nak}$ has a density w.r.t. $\Lambda^{dk}$, provided that for all $i \geq 1$, $X_{i}$ has a positive one w.r.t. $\Lambda^{d}$. Finally, our version of the GCP, which is our main result (Theorem \ref{mainTheo} in Section \ref{sec:2}), is a wide generalization of Theorem \ref{DFtheo}.

\begin{lem}\label{densCond}
Let $(X_{i})_{i \geq 1}$ be a sequence of i.n.n.i.d. r.v.'s. Assume that that for all $i \geq 1$, $X_{i}$ has a positive density $p_{i}$ w.r.t. $\Lambda^{d}$. For any $n \geq 1$, let $J_{n} \subset \left\{ 1, ..., n \right\}$ with $\left|J_{n}\right|<n$. Set $L_{n} := \left\{ 1, ..., n \right\} \setminus J_{n}$ and $S_{L_{n}} := \sum\limits_{i \in L_{n}} X_{i}$. Then, for $s \in \mathbb{R}^{d}$, the distribution $Q_{s,J_{n}}$ of $(X_{i})_{i \in J_{n}}$ given $\left\{S_{1,n}=s\right\}$ has a density $q_{s,J_{n}}$ w.r.t. $\Lambda^{d\left|J_{n}\right|}$, defined for $x=(x_{i})_{i \in J_{n}} \in (\mathbb{R}^{d})^{\left|J_{n}\right|}$ by 
\begin{equation}\label{densCondFormule}
q_{s,J_{n}}(x) = \frac{ \left[\prod\limits_{i \in J_{n}} p_{i}(x_{i}) \right]
p_{S_{L_{n}}}\left(s-\sum\limits_{i \in J_{n}} x_{i} \right) }
{p_{S_{1,n}}\left(s\right)},  
\end{equation}

\noindent
where $p_{S_{L_{n}}}$ and $p_{S_{1,n}}$ are respectively the densities of $S_{L_{n}}$ and $S_{1,n}$ w.r.t. $\Lambda^{d}$. 
\end{lem}

\begin{proof}
This follows from the change of variable formula. 
\end{proof}

\noindent
This paper is organized as follows. The main result is stated in Section \ref{sec:2}. The sequel is devoted to its proof, given in Section \ref{sec:4}, after some preliminary results in Section \ref{sec:3}. The proof of some technical lemmas are deferred to the Appendix.

\section{Main result and Examples}
\label{sec:2}

\subsection{Exponential tilting}

For any r.v. $X$, its moment generating function (mgf) $\Phi_{X}$ is defined for $\theta \in \mathbb{R}^{d}$ by
\begin{equation*}
\Phi_{X}(\theta)=\mathbb{E} \left[ \exp \langle \theta, X \rangle \right] \in (0,\infty].     
\end{equation*}

\noindent
Set $\Theta_{X}:=\left\{ \theta \in \mathbb{R}^{d} : \Phi_{X}(\theta)<\infty \right\}$, which is convex and called the \textit{domain} of $\Phi_{X}$. Let $\kappa_{X}:=\log \Phi_{X}$ be the cumulant-generating function (cgf) of $X$. Before introducing the exponential tilting below, we recall the following fundamental result on $\kappa_{X}$. 

\begin{theo}\label{kapPte}
$($Theorem 7.1. in \cite{Barndorff-Nielsen 2014}$)$ $\kappa_{X}$ is a closed convex function on $\mathbb{R}^{d}$ and is strictly convex on $\Theta_{X}$ if $P_{X}$ is not concentrated on an affine subspace of $\mathbb{R}^{d}$.
\end{theo}

\begin{defi}
Let $X$ be a r.v. and $\theta \in \Theta_{X}$. The exponentially tilted measure $\widetilde{P}^{\theta}_{X}$ is defined by 
\begin{equation*}
\frac{d\widetilde{P}^{\theta}_{X}}{dP_{X}}(x) = 
\frac{\exp \langle \theta ,x \rangle}{\Phi_{X}(\theta)}, \enskip \textrm{for } x \in \mathbb{R}^{d}.    
\end{equation*}

\noindent
Denote by $\widetilde{X}^{\theta}$ a r.v. whose probability distribution is $\widetilde{P}^{\theta}_{X}$. 
\end{defi}

\begin{rem}
The Gibbs measure $\gamma^{a}$ appearing in $(\ref{gammaa})$ is $\widetilde{P}^{\theta^{a}}_{f(X_{1})}$. 
\end{rem}

\begin{lem} \label{ppteTilting}
Let $X$ and $Y$ be i.n.n.i.d. r.v.'s which have densities w.r.t. $\Lambda^{d}$. Then, for all $\theta \in \Theta_{X}$, 
\begin{equation*}
\mathbb{E}\left[\widetilde{X}^{\theta}\right] = \nabla \kappa_{X}(\theta)
\quad \textrm{and} \quad 
\mathrm{Cov}\left(\widetilde{X}^{\theta}\right) = \mathrm{Hess}(\kappa_{X})(\theta).
\end{equation*}

\noindent
Furthermore, for all $\theta \in \Theta_{X} \cap \Theta_{Y}$, ~$\mathbb{E}\left[ \widetilde{X+Y}^{\theta} \right]=\mathbb{E}\left[ \widetilde{X}^{\theta}+\widetilde{Y}^{\theta} \right]$. 
\end{lem}

\begin{proof}
This follows from the change of variable formula. 
\end{proof}

\noindent\\
We will assume that the $(\Phi_{X_{i}})_{i \geq 1}$ have a common domain $\Theta$. \textit{Throughout the sequel, for any $\theta \in \Theta$, $\left(\widetilde{X}_{i}^{\theta}\right)_{i \geq 1}$ is assumed to be a sequence of i.n.n.i.d. r.v.'s defined on a same probability space}.

\noindent\\
For any $n \geq 1$, since the conditioning event defining $Q_{nak}$ is 
$\left\{ S_{1,n}=na \right\}$, we search $\theta \in \Theta$ such that $\mathbb{E}\left[ \widetilde{S}_{1,n}^{\theta} \right]=na$, which is (by Lemma \ref{ppteTilting}) equivalent to 
\begin{equation}\label{equaTilting}
\nabla \overline{\kappa}_{n}(\theta)=a, 
\quad \textrm{where } 
\overline{\kappa}_{n}:=\frac{1}{n} \sum\limits_{i=1}^{n} \nabla \kappa_{i}.
\end{equation}

\noindent
We will prove in Section \ref{exisTilted} that, under suitable assumptions, 
$(\ref{equaTilting})$ has a unique solution denoted by $\theta_{n}^{a}$, which defines the limiting distribution of $Q_{nak}$.

\subsection{Notations}

\begin{itemize}

\item For any $D \geq 1$ and any set $E \subset \mathbb{R}^{D}$, we denote respectively by $\mathrm{int}(E)$, $\mathrm{c\ell}(E)$ and $\mathrm{conv}(E)$ the interior, the closure and the convex hull of $E$.  

\item Assuming that the $(X_{i})_{i \geq 1}$ have a common support $S_{X}$, set 
\begin{equation*}
C_{X}:=\mathrm{c\ell}(\mathrm{conv}(S_{X}))    
\end{equation*}

\item For $v=(v_{i})_{1 \leq i \leq D} \in \mathbb{R}^{D}$, $\left\| v \right\|$ denotes its euclidean norm. Set $\left\| v \right\|_{\infty}:=\max\limits_{1 \leq i \leq D} |v_{i}|$. If $M$ is a $D \times D$ matrix, we still denote by $\left\| M \right\|$ and $\left\| M \right\|_{\infty}$ the subordinate norms. 

\vspace{.2cm}

\item For any symmetric matrix $C$, denote by $\lambda_{\min}(C)$ (resp. $\lambda_{\max}(C)$) its smallest (resp. largest) eigenvalue. 

\vspace{.2cm}

\item For all $i \geq 1$, set $\Phi_{i}:=\Phi_{X_{i}}$ ~;~ $\kappa_{i} :=\kappa_{X_{i}}$ and $m_{i} := \nabla \kappa_{i}$.

\item For all $i \geq 1$ and $\theta \in \Theta$, set $\widetilde{C}_{i}^{\theta} := \mathrm{Cov}\left(\widetilde{X_{i}}^{\theta}\right)$. Let $\widetilde{p}_{i}^{\theta}$ be the density of $\widetilde{X}_{i}^{\theta}$ and $\widetilde{\xi}_{i}^{\theta}$ be its characteristic function. 

\item Let $\Sigma$ be a sub $\sigma$-algebra of $\mathcal{B}^{D}$. For all probability measures $P$ and $Q$ on $\mathbb{R}^{D}$, the total variation distance between $P$ and $Q$ on $\Sigma$ is denoted by 
\begin{equation*}
d_{TV}^{\Sigma}\left(P;Q\right) := 2\sup\limits_{B \in \Sigma} \left| P(B) - Q(B) \right|. 
\end{equation*}

\end{itemize}

\subsection{Assumptions}

We list below all the assumptions of our main result. 

\subsubsection{General assumptions}

$\left( \mathcal{S}upp \right)$ : The $(X_{i})_{i \geq 1}$ have a common support $S_{X}$.

\noindent\\
$\left( \mathcal{D}en \right)$ : For all $i \geq 1$, $X_{i}$ has a positive density $p_{i}$ w.r.t. the Lebesgue measure.

\subsubsection{Assumptions on mgf's and cgf's}

$\left(\mathcal{D}om \right)$ : The $\Phi_{i}$'s have a common domain $\Theta$, so that for all $i \geq 1$, $\Theta_{i}=\Theta$.  

\noindent\\
$\left( \mathcal{L}gt \right)$ : For all $i \geq 1$,  $X_{i}$ has light tails, in the sense that $\mathrm{int}(\Theta) \neq \varnothing$.

\noindent\\
$\left(\mathcal{S}tp \right)$ : For all $i \geq 1$, $\kappa_{i}$ is steep (see Definition \ref{defiSteep} in section \ref{exisTilted}). 

\noindent\\
$\left(\mathcal{S}tc \right)$ : For all $i \geq 1$, $\kappa_{i}$ is strictly convex on $\Theta$.

\noindent\\
$\left(\mathcal{B}d\theta \right)$ : For all $a \in \mathrm{int}(C_{X})$, there exists a compact $K^{a} \subset \mathrm{int}(\Theta)$ such that 
\begin{equation*}
\left\{\theta_{n}^{a} : n \geq 1\right\} \subset K^{a}. 
\end{equation*}

\subsubsection{Assumptions on the tilted densities}\label{asTilt}

$\left( \mathcal{C}v \right)$ : For any compact $K \subset \mathrm{int}(\Theta)$, 
\begin{equation*}
0 < \lambda_{\min}^{K} \leq \lambda_{\max}^{K} < \infty,
\end{equation*}

\noindent
where $\lambda_{\min}^{K} := \inf\limits_{i \geq 1} ~\inf\limits_{\theta \in K} \lambda_{\min}\left(\widetilde{C}_{i}^{\theta}\right)$ and $\lambda_{\max}^{K} := \sup\limits_{i \geq 1} ~\sup\limits_{\theta \in K} \enskip \lambda_{\max}\left(\widetilde{C}_{i}^{\theta}\right)$.

\noindent\\
$\left( \mathcal{A}m4 \right)$ : For any compact $K \subset \mathrm{int}(\Theta)$,   
\begin{equation*}
\sup\limits_{i \geq 1} ~\sup\limits_{\theta \in K} 
~\mathbb{E} \left[ \left\| \widetilde{X}^{\theta}_{i} - m_{i}(\theta) \right\|^{4} \right] < \infty.
\end{equation*}

\noindent\\
$(\mathcal{C}f1)$ : For any compact $K \subset \mathrm{int}(\Theta)$, there exist positive constants $\delta_{K}, C_{K}$ and $R_{K}$, satisfying:
\begin{equation*}
~\sup\limits_{i \geq 1} ~\sup\limits_{\theta \in K}  ~\sup\limits_{\|t\| \geq R_{K}} ~
\left|\widetilde{\xi}^{\theta}_{i}(t)\right| \leq \frac{C_{K}}{\|t\|^{\delta_{K}}}.  
\end{equation*}

\noindent\\
$(\mathcal{C}f2)$ : For any compact $K \subset \mathrm{int}(\Theta)$, for all $\beta > 0$, 
\begin{equation*}
~\sup\limits_{i \geq 1} ~\sup\limits_{\theta \in K} 
~\sup\limits_{\|t\| \geq \beta}
~\left| \widetilde{\xi}^{\theta}_{i}(t) \right| < 1. 
\end{equation*}

\subsection{Discussion on assumptions}

\begin{rem}
Theorem \ref{kapPte} gives a sufficient condition for $\left(\mathcal{S}tc \right)$. 
\end{rem}

\begin{rem}
For any $a \in int(C_{X})$, the assumptions on the tilted densities are applied to a compact $K^{a}$ deriving from $\left( \mathcal{B}d\theta \right)$. This avoids to check assumptions requiring the expression of $\theta_{n}^{a}$. However, Remark 
$\ref{remWeakAs}$ in Section $\ref{sec:3}$ provides weaker assumptions, involving the $\left(\widetilde{X}_{i}^{\theta_{n}^{a}}\right)$, for which $\left( \mathcal{B}d\theta \right)$ is not necessary. When $d=1$, $\left( \mathcal{B}d\theta \right)$ can be replaced by a natural uniformity assumption (see Section $\ref{sectionD1}$), denoted by $(\mathcal{U}\kappa)$. When $d>1$, sufficient conditions for $\left( \mathcal{B}d\theta \right)$ are given in Remark $\ref{condBDtheta}$.
\end{rem}

\begin{rem}
In order to insure that $(\ref{DFconcluTheo})$ holds uniformly in $a$ in Theorem $\ref{DFtheo}$, the assumptions on the tilted densities in \cite{Diaconis and Freedman 1988} are of the form $\sup\limits_{\theta \in \Theta} F(\theta)<\infty$. Such assumptions, involving the whole $\Theta$, are not easy to check (see the Examples in \cite{Diaconis and Freedman 1988}). For our main result, we do not require this uniformity in $a$. So, under $\left( \mathcal{B}d\theta \right)$, our analogue assumptions are of the form $\sup\limits_{i \geq 1} \sup\limits_{\theta \in K} F_{i}(\theta)<\infty$, where $K$ is compact. Actually, for all $i \geq 1$, $F_{i}$ is continuous, so that one only needs to check the finiteness of a $\sup\limits_{i \geq 1}$.  
\end{rem}

\begin{rem}
$(\mathcal{C}f1)$ is a classical assumption on Fourier transforms, related to the regularity of $\widetilde{p}_{i}^{\theta}$. For example, when $d=1$, it holds if $\widetilde{p}_{i}^{\theta}$ are functions of uniformly $($w.r.t. $i$ and $\theta)$ bounded variations. $(\mathcal{C}f2)$ is also a smoothness assumption since it is the analogue of Condition $1.3.$ in \cite{Diaconis and Freedman 1988}. Therefore, as explained in \cite{Diaconis and Freedman 1988}, it means that for all $i \geq 1$, $p_{i}$ is not concentrated near a lattice. 
\end{rem}

\subsection{Main result}

\begin{theo}\label{mainTheo}
Assume that $\left( \mathcal{S}upp \right)$, $\left( \mathcal{D}en \right)$, $\left( \mathcal{D}om \right)$, $\left( \mathcal{L}gt \right)$, $\left( \mathcal{S}tp \right)$,  $\left( \mathcal{S}tc \right)$, $\left(\mathcal{B}d\theta \right)$, $\left( \mathcal{C}v \right)$, $\left( \mathcal{A}m4 \right)$, $(\mathcal{C}f1)$ and $(\mathcal{C}f2)$ hold. 
If $k=o(n)$, then, for any $a \in \mathrm{int}(C_{X})$, 
\begin{equation}\label{myDF}
d_{TV}^{\mathcal{B}^{dk}} \left( Q_{nak} ; \widetilde{P}_{1:k}^{\theta_{n}^{a}} \right) = O\left(\frac{k}{n}\right), 
\end{equation}

\noindent
where $\widetilde{P}_{1:k}^{\theta_{n}^{a}}$ is the joint distribution of the  independent r.v.'s $\left(\widetilde{X}_{i}^{\theta_{n}^{a}}\right)_{1 \leq i \leq k}$.  
\end{theo}

\begin{rem}
$(\ref{myDF})$ provides the rate of convergence to the limiting distribution. Here,  uniformity in $a$ does not hold, contrarily to Theorem $\ref{DFtheo}$. In \cite{Diaconis and Freedman 1988}, this uniformity is necessary for its applications. 
\end{rem}

\begin{rem}
$(\ref{myDF})$ is interpreted in Statistical Mechanics as follows $($see \cite{Khinchin 1949}$)$. Let $G$ be a system of $n$ independent non-identical components, which are of a purely energy nature. $G$ is divided in two parts: a small one $G_{1}$, composed of $k=o(n)$ components, immersed in $G_{2}$. Assume that $G$ is isolated, so that the total energy $E$ is fixed, equal to $na$. Then, $(\ref{myDF})$ describes the distribution of energy, as $n \rightarrow \infty$, of $G_{1}$ under the constraint $\left\{E=na\right\}$, when thermal equilibrium is reached $($when $G_{1}$ and $G_{2}$ have the same temperature$)$. $\theta_{n}^{a}$ is interpreted as a quantity called the inverse temperature. 
\end{rem}

\subsection{The one-dimensional case}\label{sectionD1}

Let $(X_{i})_{i \geq 1}$ be a sequence of independent r.v.'s valued in $\mathbb{R} ~(d=1)$. Under $(\mathcal{S}upp)$ and $(\mathcal{D}om)$, $\mathrm{int}(\Theta)$ and $\mathrm{int}(C_{X})$ are open intervals: $\mathrm{int}(\Theta) = (\alpha, \beta)$ and $\mathrm{int}(C_{X}) = (A,B)$, for some $\alpha, \beta, A, B \in [-\infty, \infty]$.   

\begin{defi}
Let $f : (\alpha, \beta) \longrightarrow (A,B)$ be a differentiable function. Consider the following properties:

\noindent
$(\mathcal{H})$ : For all $\theta \in (\alpha, \beta)$, ~$\frac{df}{d\theta}(\theta) > 0$ 
\quad and \quad
$\lim\limits_{\theta \rightarrow \alpha} f(\theta) = A$ \enskip ; \enskip 
$\lim\limits_{\theta \rightarrow \beta} f(\theta) = B$.

\noindent\\
$(\mathcal{H}\kappa)$ : For all $i \geq 1$, $m_{i}:=\frac{d\kappa_{i}}{d\theta}$ satisfies $(\mathcal{H})$. 

\noindent\\
$(\mathcal{U}\kappa)$ : There exist functions $f_{+}$ and $f_{-}$ which satisfy $(\mathcal{H})$, such that for all $i \geq 1$ and $\theta \in (\alpha, \beta)$, 
\begin{equation}\label{condBD1}
f_{-}(\theta) \leq m_{i}(\theta) \leq f_{+}(\theta).
\end{equation}
\end{defi}

\begin{lem}\label{lemEX1}
Assume that $(\mathcal{H}\kappa)$ and $(\mathcal{U}\kappa)$ hold. Then, for all $a \in (A,B)$ and $n \geq 1$, there exists a unique $\theta_{n}^{a} \in (\alpha, \beta)$ such that 
\begin{equation}\label{tiltExist1}
\frac{d\overline{\kappa}_{n}}{d\theta}(\theta_{n}^{a}) = a, 
\end{equation}

\noindent
and the sequence $(\theta_{n}^{a})_{n \geq 1}$ is bounded. 
\end{lem}

\begin{proof}
Any function satisfying $(\mathcal{H})$ induces a homeomorphism from $(\alpha, \beta)$ to $(A,B)$ and $(\mathcal{H})$ is stable under convex combinations of functions. We deduce  the existence of a unique $\theta_{n}^{a} \in (\alpha, \beta)$ such that $(\ref{tiltExist1})$ holds. Now, by $(\ref{condBD1})$, for any $i \geq 1$ and $n \geq 1$,  
\begin{equation*}
f_{-}(\theta_{n}^{a}) \leq m_{i}(\theta_{n}^{a}) \leq f_{+}(\theta_{n}^{a}). 
\end{equation*}

\noindent
So, for all $n \geq 1$, $f_{-}(\theta_{n}^{a}) \leq \frac{d\overline{\kappa}_{n}}{d\theta} (\theta_{n}^{a}) = a \leq f_{+}(\theta_{n}^{a})$, which implies that 
\begin{equation*}
(f_{+})^{-1}(a) \leq \theta_{n}^{a} \leq (f_{-})^{-1}(a). 
\end{equation*}
\end{proof}

\begin{corr}
When $d=1$, among the assumptions of Theorem \ref{mainTheo}, one can replace $(\mathcal{S}tp)$ and $(\mathcal{B}d\theta)$ by $(\mathcal{H}\kappa)$ and $(\mathcal{U}\kappa)$, and keep the other ones. Then the conclusion still holds. 
\end{corr}

\subsection{Examples}

\subsubsection{Normal distribution}

For all $i \geq 1$, $X_{i}$ is a r.v. with normal distribution $\mathcal{N}_{d}\left(\mu_{i},\Gamma_{i} \right)$ on $\mathbb{R}^{d}$. Assume that 
\begin{equation}\label{normalCond}
\sup\limits_{i \geq 1} \left\| \mu_{i} \right\| < \infty \quad \textrm{and} \quad 
0 < \inf\limits_{i \geq 1} \lambda_{min}(\Gamma_{i}) \leq \sup\limits_{i \geq 1} \lambda_{max}(\Gamma_{i}) < \infty. 
\end{equation}

\noindent
Clearly, $\left( \mathcal{S}upp \right)$, $\left( \mathcal{D}en \right)$, $\left( \mathcal{D}om \right)$ and  $\left( \mathcal{L}gt \right)$ hold. Recall that, for all $i \geq 1$ and \\
$\theta \in \Theta = \mathbb{R}^{d}$, $\kappa_{i}(\theta) = \langle \mu_{i},\theta \rangle + \frac{1}{2}\theta^{t}\Gamma_{i}\theta$, so that for all $\theta \in \mathbb{R}^{d}$, 
\begin{equation*}
\mathbb{E}\left[\widetilde{X}_{i}^{\theta}\right] = 
\nabla \kappa_{i}(\theta) = \mu_{i}+\Gamma_{i}\theta
\quad \textrm{and} \quad 
\mathrm{Cov} \left(\widetilde{X}_{i}^{\theta}\right) = \mathrm{Hess}(\kappa_{i})(\theta) = \Gamma_{i}.
\end{equation*}

\noindent
So, by $(\ref{normalCond})$, $(\mathcal{C}v)$ holds. By Theorem \ref{kapPte}, $\left( \mathcal{S}tc \right)$ is satisfied. Clearly, $(\mathcal{S}tp)$ holds. Set $\overline{\mu}_{n} := \frac{1}{n} \sum\limits_{i=1}^{n} \mu_{i}$ and 
$\overline{\Gamma}_{n}:=\frac{1}{n} \sum\limits_{i=1}^{n} \Gamma_{i}$. After easy calculations, we get that for any $a \in int(C_{X})$ and $n \geq 1$, the equation $\nabla \overline{\kappa}_{n} (\theta) = a$ is equivalent to 
\begin{equation}
\label{normalEqua}
\left( \overline{\Gamma}_{n}\right) \theta = a - \overline{\mu}_{n}. 
\end{equation}

\noindent
Then, $(\ref{normalCond})$ implies that for all $n \geq 1$, $(\ref{normalEqua})$ defines a unique $\theta_{n}^{a}$ and that $(\mathcal{B}d\theta)$ holds. Finally, it is well known that for all $i \geq 1$ and $\theta \in \mathbb{R}^{d}$,
\begin{equation*}
\widetilde{X}_{i}^{\theta} \sim \mathcal{N}\left(\mu_{i}+\Gamma_{i}\theta,\Gamma_{i}\right). 
\end{equation*}

\noindent
We deduce from $(\ref{normalCond})$ that $\left( \mathcal{A}m4 \right)$, $(\mathcal{C}f1)$ and $(\mathcal{C}f2)$ hold.

\subsubsection{Gamma distribution}

Fix $\beta>0$. For any $i \geq 1$, let $X_{i}$ be a r.v. valued in $\mathbb{R} ~(d=1)$, with distribution $\Gamma(\alpha_{i}, \beta)$, defined by the density
\begin{equation*}
p_{i}(x) = \frac{\beta^{\alpha_{i}}}{\Gamma(\alpha_{i})}x^{\alpha_{i}-1} \exp(-\beta x), \quad x>0.
\end{equation*}

\noindent
Assume that 
\begin{equation}\label{GammaShapeB}
0 < \alpha_{-} := \inf\limits_{i \geq 1} \alpha_{i} \leq \alpha_{+} := \sup\limits_{i \geq 1} \alpha_{i} < \infty. 
\end{equation}

\noindent
Clearly, $(\mathcal{S}upp)$ and $(\mathcal{D}om)$ hold, with $S_{X}=(0,\infty)$ and $\Theta = (-\infty,\beta)$. Indeed, for all $i \geq 1$ and $\theta \in (-\infty,\beta)$, $\Phi_{i}(\theta) = \left(1-\frac{\theta}{\beta} \right)^{-\alpha_{i}}$. Furthermore, 
\begin{equation*}
\kappa_{i}(\theta)=-\alpha_{i}\log\left( 1-\frac{\theta}{\beta} \right)
\quad ; \quad
\frac{d\kappa_{i}}{d\theta}(\theta) = \frac{\alpha_{i}}{\beta-\theta}
\quad ; \quad 
\frac{d^{2}\kappa_{i}}{d\theta^{2}}(\theta) = \frac{\alpha_{i}}{(\beta-\theta)^{2}}.
\end{equation*}

\noindent
$(\mathcal{H}\kappa)$ and $(\mathcal{U}\kappa)$ hold, since for all $i \geq 1$ and $\theta \in (-\infty,\beta)$,
\begin{equation}
\frac{d\kappa_{i}}{d\theta} \textrm{ satisfies } (\mathcal{H})
\quad \textrm{and} \quad
\frac{\alpha_{-}}{\beta-\theta} \leq \frac{d\kappa_{i}}{d\theta}(\theta)
\leq \frac{\alpha_{+}}{\beta-\theta}. 
\end{equation}

\noindent
The other assumptions obviously hold. Now, for all $i \geq 1$ and $\theta \in (-\infty,\beta)$, 
\begin{equation*}
\widetilde{X}_{i}^{\theta} \sim \Gamma(\alpha_{i},\beta-\theta).
\end{equation*}

\noindent
This, combined to $(\ref{GammaShapeB})$, implies readily that $(\mathcal{A}m4)$, $(\mathcal{C}f1)$, $(\mathcal{C}f2)$ hold.

\section{Preliminary Results}
\label{sec:3}

\subsection{The tilted density}\label{exisTilted}

In this section, we obtain, for all $n \geq 1$, the existence and unicity of $\theta_{n}^{a}$ (see Corollary \ref{LemExistTilting} below). Then, we give sufficient conditions for $(\mathcal{B}d\theta)$, when $d \geq 1$.

\subsubsection{Pairs of Legendre type}

\begin{defi}\label{defiSteep}
Let $f$ be a convex function on $\mathbb{R}^{d}$. Set $\mathrm{dom}(f) := \left\{ x \in \mathbb{R}^{d} : f(x) < \infty \right\}$. Assume that $\mathrm{int}(\mathrm{dom}(f)) \neq \varnothing$ and $f$ is differentiable throughout $\mathrm{int}(\mathrm{dom}(f))$. Then, for any boundary point $x$ of $\mathrm{dom}(f)$, we say that $f$ is \textit{steep} at $x$ if 
\begin{equation}\label{gradSteep}
\left\| \nabla f(x_{i}) \right\|  \longrightarrow \infty 
\end{equation}

\noindent
whenever $(x_{i})$ is a sequence of points in $\mathrm{int}(\mathrm{dom}(f))$ converging to $x$. Furthermore, $f$ is called \textit{steep} if it is steep at all boundary point of $\mathrm{dom}(f)$. 
\end{defi}

\begin{defi}\label{legendre}
Let $C$ be an open convex set. Let $f$ be a strictly convex and differentiable function on $C$ such that $(\ref{gradSteep})$ holds  whenever $(x_{i})$ is a sequence of points in $C$ converging to a boundary point of $C$. Then the pair $(C,f)$ is said to be of Legendre type.  
\end{defi}

\begin{lem}\label{factSteep}
Assume that $(\mathcal{D}om)$, $\left(\mathcal{L}gt \right)$, $\left(\mathcal{S}tp \right)$ and $\left(\mathcal{S}tc \right)$ hold. For $n \geq 1$, set 
\begin{equation*}
\overline{\kappa}_{n} := \frac{1}{n} \sum\limits_{i=1}^{n} \kappa_{i}.
\end{equation*}

\noindent
Then, for all $n \geq 1$, the pair $(\mathrm{int}(\Theta), \overline{\kappa}_{n})$ is of Legendre type.   
\end{lem}

\begin{proof}
By $\left(\mathcal{L}gt \right)$, $\mathrm{int}(\Theta) \neq \emptyset$. The assumptions imply readily that for all $i \geq 1$, $\kappa_{i}$ is \textit{essentially smooth} and \textit{essentially strictly convex} (see p.87-88 in \cite{Barndorff-Nielsen 2014}), from which it follows that the pair $(\mathrm{int}(\Theta),\kappa_{i})$ is of Legendre type. We deduce that $(\mathrm{int}(\Theta),\overline{\kappa}_{n})$ is also of Legendre type. In particular, the steepness of $\overline{\kappa}_{n}$ is derived from Theorem 5.27 in \cite{Barndorff-Nielsen 2014}, a characterization of steepness which is stable under convex combinations of functions.     
\end{proof}

\begin{rem}
We have actually proved that the set $\Gamma(\mathrm{int}(\Theta))$ defined below is convex: 
\begin{equation*}
\Gamma(\mathrm{int}(\Theta)):=\left\{ f : \textrm{the pair } (\mathrm{int}(\Theta),f) \textrm{ is of Legendre type} \right\}.
\end{equation*}
\end{rem}

\subsubsection{Existence and properties of $\theta^{a}_{n}$}

\begin{defi}
Let $f$ be a convex function on $\mathbb{R}^{d}$. Its conjugate function is defined on $\mathbb{R}^{d}$ by 
\begin{equation}
f^{*}(a) = \sup\limits_{x \in \mathbb{R}^{d}}\left\{ \langle x, a \rangle-f(x) \right\}
\end{equation}
\end{defi}

\begin{lem}\label{lemSteep}
Assume that $\left(\mathcal{S}upp\right)$ holds. Then, for any $n \geq 1$,
\begin{equation*}
\mathrm{int}(\mathrm{dom}((\overline{\kappa}_{n})^{*})) = \mathrm{int}(C_{X}), 
\end{equation*}

\noindent
where $\mathrm{dom}((\overline{\kappa}_{n})^{*})=\left\{ a \in \mathbb{R}^{d} : (\overline{\kappa}_{n})^{*}(a)<\infty \right\}$. 
\end{lem}

\begin{proof}
The proof is given in Appendix. 
\end{proof}

\begin{theo}\label{theoLegendre} $($Theorem 5.33. in \cite{Barndorff-Nielsen 2014}$)$ Let $f$ be a convex and lower semi-continuous function on $C:=\mathrm{int}(\mathrm{dom}(f))$. Set $C^{*}:=\mathrm{int}(\mathrm{dom}(f^{*}))$. If the pair $(C,f)$ is of Legendre type, then the gradient mapping $\nabla f$ is a homeomorphism from $C$ onto $C^{*}$, and $\nabla f^{*} = (\nabla f)^{-1}$.  
\end{theo}

\begin{corr}\label{LemExistTilting}
Assume that $\left(\mathcal{S}upp\right)$,  $(\mathcal{D}om)$, $\left( \mathcal{L}gt \right)$, $\left( \mathcal{S}tp \right)$ and $\left( \mathcal{S}tc \right)$ hold. Then, for any $n \geq 1$ and $a \in \mathrm{int}(C_{X})$, there exists a unique $\theta^{a}_{n} \in \mathrm{int}(\Theta)$ such that   
\begin{equation}
\nabla \overline{\kappa}_{n}(\theta^{a}_{n}) = a. 
\end{equation}

\noindent 
Namely, for any $n \geq 1$ and $a \in \mathrm{int}(C_{X})$,
\begin{equation}\label{conjKappaN}
\theta^{a}_{n}=\nabla (\overline{\kappa}_{n})^{*}(a).   
\end{equation}
\end{corr}

\begin{proof}
Fix $n \geq 1$. We aim to apply Theorem \ref{theoLegendre} to $\overline{\kappa}_{n}$. First, by $(\mathcal{D}om)$, $\mathrm{int}(\mathrm{dom}(\overline{\kappa}_{n}))=\mathrm{int}(\Theta)$ and by Theorem \ref{kapPte}, $\overline{\kappa}_{n}$ is lower semi-continuous on $\mathrm{int}(\Theta)$. Then, Lemma $\ref{factSteep}$, Theorem $\ref{theoLegendre}$ and Lemma $\ref{lemSteep}$ imply that the gradient mapping $\nabla \overline{\kappa}_{n}$ is a homeomorphism from $\mathrm{int}(\Theta)$ to $\mathrm{int}(C_{X})$. 
\end{proof}

\begin{rem}\label{condBDtheta}
By $(\ref{conjKappaN})$, one can sometimes compute $\nabla (\overline{\kappa}_{n})^{*}$ for any $n \geq 1$, and check directly $(\mathcal{B}d\theta)$. A sufficient condition for $(\mathcal{B}d\theta)$ is given by Proposition $\ref{UkapD}$ below, which is an analogue of $(\mathcal{U}\kappa)$ when $d>1$. However, it is more difficult to check in practice than $(\mathcal{U}\kappa)$.
\end{rem}

\noindent
Set $\nabla \Gamma(\mathrm{int}(\Theta)):=\left\{ \nabla f : f \in \Gamma(int(\Theta)) \right\}$. For functions $f,g$ from $\mathrm{int}(\Theta)$ to $\mathrm{int}(C_{X})$, set 
\begin{equation*}
d_{U}(f;g) := \sup \left\{ \left\| f(\theta)-g(\theta) \right\| : \theta \in \mathrm{int}(\Theta) \right\} \in [0,\infty].
\end{equation*}

\begin{prop}\label{UkapD}
Set $N_{\kappa} := \left\{\nabla\kappa_{i} : i \geq 1 \right\}$. Assume that $\mathrm{conv}(N_{\kappa})$ is totally bounded in $\nabla\Gamma(\mathrm{int}(\Theta))$ in the following sense: For any $\eta>0$, there exists a finite number of functions $(\phi_{j})_{j \in J}$ in $\nabla \Gamma(\mathrm{int}(\Theta))$ such that 
\begin{equation}\label{NkappaCover}
\mathrm{conv}(N_{\kappa}) \subset \bigcup\limits_{j \in J} B_{U}(\phi_{j},\eta), 
\end{equation}

\noindent
where $B_{U}(\phi_{j},\eta):=\left\{ f \in \nabla\Gamma(\mathrm{int}(\Theta)) : d_{U}(\phi_{j},f)<\eta \right\}$. Then, $(\mathcal{B}d\theta)$ holds.
\end{prop}

\begin{proof}
Let $\eta>0$ and $(\phi_{j})_{j \in J}$, with $\left| J \right|<\infty$, such that $(\ref{NkappaCover})$ holds. For all $n \geq 1$, $\nabla \overline{\kappa}_{n} \in \mathrm{conv}(N_{\kappa})$, so that there exists $\phi_{j(n)} \in (\phi_{j})_{j \in J}$ such that $d_{U}\left( \nabla \overline{\kappa}_{n};\phi_{j(n)} \right) <\eta$. So, for any $a \in \mathrm{int}(C_{X})$ and $n \geq 1$, 
\begin{equation}\label{phiJtheta}
\left\| \phi_{j(n)}(\theta^{a}_{n}) - \nabla \overline{\kappa}_{n}(\theta^{a}_{n}) \right\| = \left\| \phi_{j(n)}(\theta^{a}_{n}) - a \right\| < \eta. 
\end{equation}

\noindent
Now, by Theorem \ref{theoLegendre}, all functions in $\nabla\Gamma(\mathrm{int}(\Theta))$ are homeomorphisms from $\mathrm{int}(\Theta)$ to $\mathrm{int}(C_{X})$. So, by $(\ref{phiJtheta})$, for sufficiently small $\eta$, $\left\| \theta^{a}_{n} - (\phi_{j(n)})^{-1}(a) \right\|<1$ and, for all $n \geq 1$, $\left\| \theta^{a}_{n} \right\|<1+\left\|(\phi_{j(n)})^{-1}(a)\right\| \leq 1+\max\limits_{j \in J} \left\|(\phi_{j})^{-1}(a)\right\|<\infty$, since $\left| J \right|<\infty$.
\end{proof}

\subsection{Sufficiency Theory}\label{SuffiTheory}

In this section, we prove that the total variation distance $2\sup\limits_{B \in \mathcal{B}^{dk}} \left| Q_{nak}(B)-\widetilde{P}_{1:k}(B) \right|$ is attained on 
$\sigma(T):=T^{-1}\left(\mathcal{B}^{d}\right) \subset \mathcal{B}^{dk}$, where 
$T : (\mathbb{R}^{d})^{k} \longrightarrow \mathbb{R}^{d}$ is the map defined by $T(x)=\sum\limits_{i=1}^{k} x_{i}$, for $x = (x_{i})_{1 \leq i \leq k} \in (\mathbb{R}^{d})^{k}$. The proof is based on Proposition $\ref{PropDHS}$ below, a reformulation of Proposition 2 in \cite{Diaconis Holmes and Shahshahani 2013} which relies on the co-area formula (see \cite{Federer 1969}). We refer to \cite{Diaconis Holmes and Shahshahani 2013} for detailed discussions on this formula.

\subsubsection{Sufficient sub $\sigma$-algebra}

\begin{defi}
Let $(E, \mathcal{A}, P)$ be a probability space and $\Sigma$ a sub $\sigma$-algebra of $\mathcal{A}$. A function $\nu_{P}^{\Sigma}$ from $E \times \mathcal{A}$ to $[0,1]$ is called a regular conditional probability $($RCP$)$ for $P$ given $\Sigma$ if 
\begin{align*}
& \forall x \in E, \nu_{P}(x, \cdot) \textrm{ is a probability measure on } \mathcal{A}. \\
& \forall A \in \mathcal{A}, \textrm{ the function }
x \mapsto \nu_{P}(x, A) \textrm{ is } 
\Sigma \textrm{ measurable}. \\
& \forall C \in \Sigma, \forall A \in \mathcal{A}, 
P(C \cap A) = \int_{C} \nu_{P}^{\Sigma}(x,A) P(dx). 
\end{align*}
\end{defi}

\begin{defi}
Let $P$ and $Q$ be p.m.'s on $(E, \mathcal{A})$. We say that $\Sigma \subset \mathcal{A}$ is sufficient w.r.t. $P$ and $Q$ if for all $A \in \mathcal{A}$, 
\begin{equation}\label{suffy}
\nu_{P}^{\Sigma}(\cdot,A) = \nu_{Q}^{\Sigma}(\cdot,A) \textrm{ a.e. }P \textrm{ and a.e. } Q. 
\end{equation}
\end{defi}

\begin{lem}\label{dVTsuffy}
Assume that $\Sigma \subset \mathcal{A}$ is sufficient w.r.t. $P$ and $Q$. Then, 
\begin{equation*}
d_{TV}^{\mathcal{A}}\left(P;Q\right) = d_{TV}^{\Sigma}\left(P;Q\right). 
\end{equation*}
\end{lem}

\begin{proof}
The proof is elementary. See Lemma (2.4) in \cite{Diaconis and Freedman 1987} for details. 
\end{proof}

\subsubsection{RCP for $P$ given $\sigma(T)$}

\begin{defi}
The $m$-dimensional \textbf{Hausdorff measure} $\overline{\mathcal{H}}^{m}$ on $\mathbb{R}^{D}$ is defined by 
\begin{equation*}
\overline{\mathcal{H}}^{m}(A) = 
\lim\limits_{\delta \rightarrow 0} 
~\inf \left\{ \sum\limits_{i \geq 1} \alpha_{m}\left( \frac{\mathrm{diam}(S_{i})}{2}\right)^{m} : A \subset \cup_{i \geq 1} S_{i} ~;~ \mathrm{diam}(S_{i}) \leq \delta \right\}, 
\quad A \subset \mathbb{R}^{D}, 
\end{equation*}

\noindent
where for any $S \subset \mathbb{R}^{D}$,  
$\mathrm{diam}(S)=\sup \left\{ \left\|x-y\right\| : x,y \in S \right\}$ and  $\alpha_{m}=\frac{\Gamma(1/2)^{m}}{\Gamma[(m/2)+1]}$ is the volume of the unit ball in $\mathbb{R}^{m}$. If $\mathcal{M}$ is any subset of $\mathbb{R}^{D}$, then $\overline{\mathcal{H}}^{m}$ defines a measure on the $\sigma$-algebra $\mathcal{B}(\mathcal{M})$, where 
\begin{equation*}
\mathcal{B}(\mathcal{M}):=\left\{ B \cap \mathcal{M} : B \in 
\mathcal{B}^{D}\right\}. 
\end{equation*}
\end{defi}

\begin{rem}
Assume that $\mathcal{M}$ is a $m$-dimensional submanifold of $\mathbb{R}^{D}$. Then, $\overline{\mathcal{H}}^{m}$ defines a measure on $\mathcal{B}(\mathcal{M})$ which coincides with the surface measure on $\mathcal{M}$. 
\end{rem}

\begin{defi}
Let $\phi: \mathbb{R}^{M} \longrightarrow \mathbb{R}^{N}$ be differentiable at $x \in \mathbb{R}^{M}$. The Jacobian $J\phi_{x}$ is defined by 
\begin{equation*}
J\phi_{x}=\left[ \det\left( D\phi_{x}(D\phi_{x})^{t} \right) \right]^{1/2}.    
\end{equation*}
\end{defi}

\bigskip 

\begin{prop}\label{PropDHS}
For $M>N \geq 1$, let $\phi: \mathbb{R}^{M} \longrightarrow \mathbb{R}^{N}$ be Lipschitz and differentiable on $\mathbb{R}^{M}$. Set $\sigma(\phi):=\phi^{-1}\left( \mathcal{B}^{N}\right)$. Let $P$ be a p.m. on $\left(\mathbb{R}^{M}, \mathcal{B}^{M}\right)$ having a density $p$ w.r.t. $\Lambda^{M}$. For $x \in \mathbb{R}^{M}$, set  $L_{x}^{\phi}:=\phi^{-1}\left(\left\{\phi(x)\right\}\right)$ and 
\begin{equation*}
\mu(x) := \displaystyle\int\limits_{L_{x}^{\phi}} \frac{p(u)}{J\phi_{u}} ~\overline{\mathcal{H}}^{M-N}(du). \end{equation*}

\noindent 
Define the map $\nu_{P}^{\sigma(\phi)} : \mathbb{R}^{M} \times \mathcal{B}^{M} \longrightarrow [0,1]$ as follows. If $\mu(x) \in \left\{0,\infty \right\}$, then set $\nu_{P}^{\phi}\left(x,A\right)=\delta_{x^{*}}(A)$, for some fixed $x^{*} \in \mathbb{R}^{M}$. Else, set  
\begin{equation*}
\nu_{P}^{\sigma(\phi)}\left(x,A\right) = \frac{1}{\mu(x)}
\displaystyle\int\limits_{L_{x}^{\phi} \cap A} 
\frac{p(u)}{J\phi_{u}} ~\overline{\mathcal{H}}^{M-N}(du). 
\end{equation*}

\noindent
Then, $\nu_{P}^{\sigma(\phi)}$ is a regular conditional probability for $P$ given $\sigma(\phi)$. 
\end{prop}

\begin{proof}
See Proposition 2 in \cite{Diaconis Holmes and Shahshahani 2013}.
\end{proof}

\begin{corr}\label{PQTsuffy}
For any $1 \leq k <n$, $a \in \mathrm{int}(C_{X})$ and $\theta \in \mathrm{int}(\Theta)$,
\begin{equation*}
d_{TV}^{\mathcal{B}^{dk}}\left( Q_{nak};\widetilde{P}_{1:k}^{\theta} \right) = 
d_{TV}^{\sigma(T)}\left( Q_{nak};\widetilde{P}_{1:k}^{\theta} \right) =
d_{TV}^{\mathcal{B}^{d}}\left( T\left(Q_{nak}\right) ; T\left(\widetilde{P}_{1:k}^{\theta} \right) \right), 
\end{equation*}

\noindent
where $\widetilde{P}_{1:k}^{\theta}:=\otimes_{i=1}^{k} \widetilde{P}_{i}^{\theta}$.
\end{corr}

\begin{proof}
Recall that $Q_{nak}$ and $\widetilde{P}_{1:k}^{\theta}$ have densities $q_{nak}$ and $\widetilde{p}_{1:k}^{\theta}$ w.r.t. $\Lambda^{dk}$, defined for $u = (u_{i})_{1 \leq i \leq k} \in (\mathbb{R}^{d})^{k}$ by 
\begin{equation*}
q_{nak}(u) = \frac{p_{1:k}(u)p_{S_{k+1,n}}(na-T(u))}
{p_{S_{1,n}}(na)}, \enskip \textrm{where } p_{1:k}(u) := \prod\limits_{i=1}^{k}p_{i}(u_{i}),
\end{equation*}

\noindent 
which follows from Lemma \ref{densCond}, and
\begin{equation*}
\widetilde{p}_{1:k}^{\theta}(u) = 
\frac{p_{1:k}(u) \exp \langle \theta,T(u) \rangle}
{\Phi_{1:k}(\theta)}, \enskip \textrm{where } \Phi_{1:k} := \prod\limits_{i=1}^{k}\Phi_{i}. 
\end{equation*}

\noindent
We use the notations of Proposition $\ref{PropDHS}$, which obviously applies with $M=dk$, $N=d$ and $\phi=T$. Notice that for all $x \in (\mathbb{R}^d)^{k}$, $L_{x}^{T}$ is a $d(k-1)$-dimensional submanifold of $\mathbb{R}^{dk}$. Now, for all $x \in (\mathbb{R}^d)^{k}$ and $u \in L_{x}^{T}$, $T(u)=T(x)$. Moreover, since $T$ is linear, its Jacobian is constant on $(\mathbb{R}^d)^{k}$. Denote by $JT \neq 0$ its value. So, for all 
$A \in \mathcal{B}\left((\mathbb{R}^{d})^{k}\right)$,
\begin{equation*}
\displaystyle\int\limits_{L_{x}^{T} \cap A} 
\frac{q_{nak}(u)}{JT_{u}} ~\overline{\mathcal{H}}^{d(k-1)}(du) = 
\frac{1}{JT} \left( \frac{p_{S_{k+1,n}}(na-T(x))}{p_{S_{1,n}}(na)} \right) \displaystyle\int\limits_{L_{x}^{T} \cap A} 
p_{1:k}(u) ~\overline{\mathcal{H}}^{d(k-1)}(du)
\end{equation*}

\noindent
and 
\begin{equation*}
\displaystyle\int\limits_{L_{x}^{T} \cap A} \frac{\widetilde{p}_{1:k}^{\theta}(u)}{JT_{u}} ~\overline{\mathcal{H}}^{d(k-1)}(du) = \frac{1}{JT} \left( \frac{\exp \langle \theta,T(x) \rangle}{\Phi_{1:k}(\theta)} \right) \displaystyle\int\limits_{L_{x}^{T} \cap A} p_{1}^{k}(u) ~\overline{\mathcal{H}}^{d(k-1)}(du).
\end{equation*}

\noindent
Consequently, when well defined, 
\begin{equation*}
\frac{\displaystyle\int\limits_{L_{x}^{T} \cap A} \frac{q_{nak}(u)}{JT_{u}} ~\overline{\mathcal{H}}^{d(k-1)}(du)}{\displaystyle\int\limits_{L_{x}^{T}} \frac{q_{nak}(u)}{JT_{u}} ~\overline{\mathcal{H}}^{d(k-1)}(du)}
~=~
\frac{\displaystyle\int\limits_{L_{x}^{T} \cap A} \frac{\widetilde{p}_{1:k}^{\theta}(u)}{JT_{u}} ~\overline{\mathcal{H}}^{d(k-1)}(du)}{\displaystyle\int\limits_{L_{x}^{T}} \frac{\widetilde{p}_{1:k}^{\theta}(u)}{JT_{u}} ~\overline{\mathcal{H}}^{d(k-1)}(du)}
~=~
\frac{\displaystyle\int\limits_{L_{x}^{T} \cap A} p_{1:k}(u) \overline{\mathcal{H}}^{d(k-1)}_{x} (du) }{\displaystyle\int\limits_{L_{x}^{T}} p_{1:k}(u) \overline{\mathcal{H}}^{d(k-1)}_{x} (du)}.
\end{equation*}

\noindent
So, by Proposition $\ref{PropDHS}$, for all $A \in \mathcal{B}\left((\mathbb{R}^{d})^{k}\right)$, 
\begin{equation*}
\nu_{Q_{nak}}^{\sigma(T)}\left(\cdot,A\right) = \nu_{\widetilde{P}_{1:k}^{\theta}}^{\sigma(T)}\left(\cdot,A\right) \textrm{ a.e. } Q_{nak} \textrm{ and a.e. } \widetilde{P}_{1:k}^{\theta}.
\end{equation*}

\noindent
Therefore, $\sigma(T)$ is sufficient w.r.t $Q_{nak}$ and $\widetilde{P}_{1:k}^{\theta}$ and we conclude by Lemma $\ref{dVTsuffy}$. 
\end{proof}

\subsection{Edgeworth expansion}

In the sequel, $\nu \in \mathbb{N}^{d}$ denotes an integral vector. For $\nu=(\nu_{i})_{1 \leq i \leq d}$, set $\left| \nu \right|:=\sum\limits_{i=1}^{d}\nu_{i}$.

\subsubsection{Weak assumptions}

For $n \geq 1$, let $L_{n} = \left\{\alpha_{n}, ..., \beta_{n}\right\}
\subset \left\{ 1, ..., n \right\}$ be a set of consecutive integers such that 
\begin{equation}\label{LnInfty}
\lim\limits_{n \rightarrow \infty} \left| L_{n} \right| = \infty. 
\end{equation}

\begin{theo}\label{generalEdge}
Let $\left\{ Z_{in} : n \geq 1 ~;~ i \in L_{n} \right\}$ be a triangular array $($t.a.$)$ of centered r.v.'s. which have densities w.r.t. $\Lambda^{d}$. Suppose that for each row of index $n \geq 1$, the $(Z_{in})_{i \in L_{n}}$ are i.n.n.i.d. Set 
\begin{equation*}
V_{L_{n}} := \frac{1}{\left|L_{n}\right|} \sum\limits_{i \in L_{n}} Cov(Z_{in}).  
\end{equation*}

\noindent
Assume that for all $n$ large enough, $V_{L_{n}}$ is positive-definite, which is equivalent to 
\begin{equation}\label{LminV}
\lambda_{min}(V_{L_{n}})>0,   
\end{equation}

\noindent
and setting $B_{L_{n}} := (V_{L_{n}})^{-1/2}$, that 
\begin{equation}\label{AM4ass}
\varlimsup\limits_{n \rightarrow \infty} \left( \frac{1}{\left|L_{n}\right|} \sum\limits_{i \in L_{n}} \mathbb{E}\left[ \left\| B_{L_{n}} Z_{in}\right\|^{4} \right] \right) < \infty. 
\end{equation}

\noindent
By $(\ref{LnInfty})$, for any $p \in \mathbb{N}$, there exists $N_{p} \in \mathbb{N}$ such that for all $n \geq N_{p}$, $\left|L_{n}\right| \geq p+1$. Suppose that there exists $p \in \mathbb{N}$ such that the functions 
$\left\{ g_{m,n}^{p} : n \geq N_{p} ~; ~ \alpha_{n}-1 \leq m \leq \beta_{n}-p \right\}$ defined by 
\begin{equation*} 
g_{m,n}^{p}(t) := \prod\limits_{i=m+1}^{m+p} \left| 
\mathbb{E}[\exp\left\{j \langle t, B_{L_{n}}Z_{in} \rangle \right\}]
\right| \quad \textrm{where } j \in \mathbb{C}, ~j^{2}=-1, 
\end{equation*}

\noindent
satisfy 
\begin{equation}\label{gammaCond}
\sup \left\{ \int g_{m,n}^{p}(t)dt : n \geq N_{p} ~; ~ \alpha_{n}-1 \leq m \leq \beta_{n}-p \right\} < \infty
\end{equation}

\noindent
and, for all $b>0$, 
\begin{equation}\label{betaCond}
\sup \left\{g_{m,n}^{p}(t) : \left\|t\right\|>b ~;~ n \geq N_{p} ~; ~ \alpha_{n}-1 \leq m \leq \beta_{n}-p \right\} < 1.
\end{equation}

\noindent\\
Then the density $q_{L_{n}}$ of $S_{L_{n}} := \left| L_{n} \right|^{-1/2} B_{L_{n}} 
\left( \sum\limits_{i \in L_{n}}Z_{in} \right)$ w.r.t. $\Lambda^{d}$ satisfies 
\begin{equation}\label{EdgeTilde}
\underset{x \in \mathbb{R}^{d}}{\sup} (1+\left\| x \right\|^{4})
\left| q_{L_{n}}(x)-\left\{ \phi(x) \left[ 1 + \left|L_{n} \right|^{-1/2} \sum\limits_{|\nu|=3} \overline{\chi}_{\nu,L_{n}} H_{3}^{(\nu)}(x) \right] \right\} \right| = 
O\left(\frac{1}{\left|L_{n}\right|}\right), 
\end{equation}

\noindent
where $\phi$ is the density of the standard normal distribution on $\mathbb{R}^{d}$, $\overline{\chi}_{\nu, L_{n}}$ is the average of the $\nu$th cumulants of $(B_{L_{n}}Z_{in})_{i \in L_{n}}$ and $H_{3}^{(\nu)}$ is a Hermite polynomial such that 
\begin{align*}
&\sum\limits_{|\nu|=3} \overline{\chi}_{\nu,L_{n}} H_{3}^{(\nu)}(x) = \frac{1}{6} \left[ \chi_{(3,0,...,0)}(x_{1}^{3}-3x_{1}) + ... + \chi_{(0,...,0,3)}(x_{d}^{3}-3x_{d}) \right] \\
& + \frac{1}{2} \left[ \chi_{(2,1,0,...,0)}(x_{1}^{2}x_{2}-x_{2}) + ... + \chi_{(0,...,0,1,2)}(x_{d}^{2}x_{d-1}-x_{d-1}) \right] \\
& + \left[ \chi_{(1,1,1,0,...,0)}(x_{1}x_{2}x_{3}) + ... + \chi_{(0,...,0,1,1),1}(x_{d}x_{d-1}x_{d-2}) \right].
\end{align*}

\noindent
Hereabove, we wrote $\chi_{\nu}$ to abbreviate $\overline{\chi}_{\nu,L_{n}}$. 
\end{theo}

\begin{proof}
This is the t.a. version of Theorem 19.3 in \cite{Bhattacharya and Rao 1976}, 
for which a sequence $(X_{i})_{i \geq 1}$ of i.n.n.i.d. r.v.'s is considered. Under assumptions generalized here to t.a.'s, it provides an approximation (as $n \rightarrow \infty$) of the density of $S_{n} := n^{-1/2}B_{n}\sum\limits_{i=1}^{n}X_{i}$, which is extended here. In the proof of that Theorem, the following function $h_{n}$ (for all $n$ large enough) is considered.   
\begin{equation*}
h_{n}(x):= x^{\alpha} \left( q_{n}(x) - \left\{ \phi(x) \left[ 1 + \left|n \right|^{-1/2} \sum\limits_{|\nu|=3} \overline{\chi}_{\nu,n} H_{3}^{(\nu)}(x) \right] \right\} \right), 
\quad x \in \mathbb{R}^{d}, 
\end{equation*}

\noindent
where $\alpha \in \mathbb{N}^{d}$ with $\left| \alpha \right| \leq 4$ and $\overline{\chi}_{\nu,n}$ is the average of the $\nu$th cumulants of $(B_{n}X_{i})_{i \geq 1}$. Let $\hat{h}_{n}$ be the Fourier transform of $h_{n}$. By the Fourier inversion theorem,   
\begin{equation*}
\underset{x \in \mathbb{R}^{d}} {\sup} \left| h_{n}(x) \right| \leq 
(2\pi)^{-d} \int \left| \hat{h}_{n}(t) \right| dt
\end{equation*}

\noindent
The aim of that proof is then to get estimates for $\displaystyle\int \left| \hat{h}_{n}(t) \right| dt$. Now, these estimates are obtained thanks to upper bounds for $\displaystyle\int \left| \hat{h}_{n}(t) \right| dt$ which hold for \textit{any fixed} $n$. Therefore, here, analogue upper bounds hold for \textit{any fixed row of index} $n$ of the t.a. 
\end{proof}

\begin{corr}
Let $(X_{i})_{i \geq 1}$ be a sequence of i.n.n.i.d. r.v.'s. Assume that $\left(\mathcal{S}upp \right)$, $\left(\mathcal{M}gf\right)$, $\left(\mathcal{L}gt \right)$, $\left(\mathcal{S}tp\right)$ and $\left(\mathcal{D}en\right)$ hold. So, for any $a \in \mathrm{int}(C_{X})$, we are allowed to consider the following t.a. 
\begin{equation}\label{taX}
\left\{ \widetilde{X}^{\theta_{n}^{a}}_{i}-m_{i}(\theta_{n}^{a}) : 
n \geq 1 ~;~ i \in L_{n} \right\}.  
\end{equation}

\noindent
If $(\ref{LminV})$, $(\ref{AM4ass})$, $(\ref{gammaCond})$
and $(\ref{betaCond})$ hold for this t.a., then  $(\ref{EdgeTilde})$ holds fot it.  
\end{corr}

\begin{proof}
The t.a. defined in $(\ref{taX})$ satisfies that each row of index $n \geq 1$ is composed of $L_{n}$ i.n.n.i.d. and centered r.v.'s. which have densities w.r.t. $\Lambda^{d}$. Therefore, this t.a. fulfills all the conditions of Theorem $\ref{generalEdge}$, which may be applied. 
\end{proof}

\subsubsection{Natural assumptions}

\textit{Throughout the sequel, in order to alleviate notations, we will write $~\widetilde{\cdot}$ instead of $~\widetilde{\cdot}^{\theta_{n}^{a}}~$}.

\begin{corr}
Let $(X_{i})_{i \geq 1}$ be a sequence of i.n.n.i.d. r.v.'s. satisfying $\left(\mathcal{S}upp \right)$, $\left(\mathcal{M}gf\right)$, $\left(\mathcal{L}gt \right)$, $\left(\mathcal{S}tp\right)$ and $\left(\mathcal{D}en\right)$, so that we may consider the t.a. defined in $(\ref{taX})$.  Assume \textbf{additionally} that $\left( \mathcal{B}d\theta \right)$, $\left( \mathcal{C}v \right)$, $\left( \mathcal{A}m4 \right)$, $\left( \mathcal{C}f1 \right)$ and $\left( \mathcal{C}f2 \right)$ hold. Then,  $(\ref{EdgeTilde})$ holds for this t.a.  
\end{corr}

\begin{proof}
By Lemmas $\ref{VdefPos}$-$\ref{BdCf2}$ below, $\left( \mathcal{B}d\theta \right)$, $\left( \mathcal{C}v \right)$, $\left( \mathcal{A}m4 \right)$, $\left( \mathcal{C}f1 \right)$ and $\left( \mathcal{C}f2 \right)$ imply that $(\ref{LminV})$, $(\ref{AM4ass})$, 
$(\ref{gammaCond})$ and $(\ref{betaCond})$ hold for the t.a. defined in $(\ref{taX})$.
\end{proof}

\begin{rem}\label{remWeakAs}
This means that the set of assumptions $\left\{ \left( \mathcal{B}d\theta \right), \left( \mathcal{C}v \right), \left( \mathcal{A}m4 \right), \left( \mathcal{C}f1 \right), \left( \mathcal{C}f2 \right) \right\}$ can be weakened by replacing it by $\left\{ (\ref{LminV}), (\ref{AM4ass}), (\ref{gammaCond}), (\ref{betaCond}) \right\}$ for the t.a. defined in $(\ref{taX})$.
\end{rem}

\bigskip

\begin{lem}\label{VdefPos}
For any $\theta \in \mathrm{int}(\Theta)$ and $n \geq 1$, set
\begin{equation*}
\widetilde{V}^{\theta}_{L_{n}} := \frac{1}{\left|L_{n}\right|} \sum\limits_{i \in L_{n}} \widetilde{C}_{i}^{\theta}.   
\end{equation*}
   
\noindent
Assume that $\left( \mathcal{B}d\theta \right)$ and 
$\left( \mathcal{C}v \right)$ hold. Then, for all $n \geq 1$, $\widetilde{V}_{L_{n}}$ is positive-definite. Set  
\begin{equation}\label{defB}
\widetilde{B}_{L_{n}} := \left( \widetilde{V}_{L_{n}} \right)^{-1/2}. 
\end{equation}

\noindent
For any $a \in \mathrm{int}(C_{X})$, let $K^{a} \subset \mathrm{int}(\Theta)$ be a compact such that $\left\{\theta^{a}_{n} : n \geq 1 \right\} \subset K^{a}$. Then, for all $x \in \mathbb{R}^{d}$ and $n \geq 1$, 
\begin{equation}\label{maxBmin}
\left[\lambda_{max}^{K^{a}}\right]^{-1/2} \left\|x\right\| \leq \left\|\widetilde{B}_{L_{n}} x\right\|
\leq \left[\lambda_{min}^{K^{a}}\right]^{-1/2} \left\|x\right\|. 
\end{equation}
\end{lem}

\begin{proof}
See Appendix.
\end{proof}

\begin{lem}\label{AM4lem}
Assume that $\left( \mathcal{B}d\theta \right)$,  
$\left( \mathcal{C}v \right)$ and 
$\left( \mathcal{A}m4 \right)$ hold. Then, 
\begin{equation*}
\varlimsup\limits_{n \rightarrow \infty} \left( \frac{1}{\left|L_{n}\right|} \sum\limits_{i \in L_{n}} \mathbb{E}\left[ \left\| \widetilde{B}_{L_{n}} \left( \widetilde{X}_{i} - m_{i}(\theta^{a}_{n}) \right) \right\|^{4} \right] \right) < \infty. 
\end{equation*}
\end{lem}

\begin{proof}
See Appendix.
\end{proof}

\noindent\\
For $p \geq 1$, define the functions 
$\left\{ \widetilde{g}_{m,n}^{p} : \alpha_{n}-1 \leq m \leq \beta_{n}-p ~; ~ n \geq N_{p} \right\}$ by 
\begin{equation*} 
\widetilde{g}_{m,n}^{p}(t) := 
\prod\limits_{i=m+1}^{m+p} \left| 
\mathbb{E} \left[\exp \left( j \langle t, \widetilde{B}_{L_{n}} \left( \widetilde{X}_{i}-m_{i}(\theta^{a}_{n}) \right) \rangle \right) \right] \right| =
 \prod\limits_{i=m+1}^{m+p} \left| 
\widetilde{\xi}_{i} \left( \widetilde{B}_{L_{n}} ~t \right) \right|.
\end{equation*}

\bigskip 

\begin{lem}\label{BdCf1}
Assume that $\left(\mathcal{B}d\theta\right)$ and $(\mathcal{C}f1)$ hold. Then, for any $p > \frac{1}{\delta_{K^{a}}}$, 
\begin{equation}\label{ConcluGamma}
\sup \left\{ \int \widetilde{g}_{m,n}^{p}(t)dt : \alpha_{n}-1 \leq m \leq \beta_{n}-p ~; ~ n \geq N_{p}\right\} < \infty.
\end{equation}
\end{lem}

\begin{proof}
See Appendix.
\end{proof}

\begin{lem}\label{BdCf2}
Assume that $\left(\mathcal{B}d\theta\right)$ and $(\mathcal{C}f2)$ hold. Then, for any $p \geq 1$, for all $b>0$, 
\begin{equation}\label{ConcluBeta}
\sup \left\{\widetilde{g}_{m,n}^{p}(t) : \left\|t\right\|>b ~;~ n \geq N_{p} ~; ~  \alpha_{n}-1 \leq m \leq \beta_{n}-p \right\} < 1.
\end{equation} 
\end{lem}

\begin{proof}
See Appendix.
\end{proof}

\section{Proof of Theorem \ref{mainTheo}}
\label{sec:4}

The proof follows essentially that of Theorem \ref{DFtheo}, but in a more general framework. In the sequel, for any r.v.'s $X$ and $Y$, if the distribution of $X$ given 
$\left\{Y=s\right\}$ has a density w.r.t. $\Lambda^{d}$, then we denote it by 
$p(X=\cdot | Y=s)$. 

\subsection{Overview of the proof}

For $1 \leq \ell \leq m$, set $S_{\ell,m} := \sum\limits_{j=\ell}^{m} X_{j}$ and $\widetilde{S}_{\ell,m} := \sum\limits_{j=\ell}^{m} \widetilde{X}_{j}$. Let $n \geq 1$ and $a \in \mathrm{int}(C_{X})$. With the notations of Corollary \ref{PQTsuffy}, $T\left(Q_{nak}\right)$ is  the distribution of $S_{1,k}$ given $\left\{S_{1,n}=na\right\}$ and $T\left(\widetilde{P}_{1:k} \right)$ is the distribution of $\widetilde{S}_{1,k}$. Then, by Corollary \ref{PQTsuffy}, 
\begin{equation}\label{suffi}
d_{TV}^{\mathcal{B}^{dk}}\left(Q_{nak} ; \widetilde{P}_{1:k}\right) = 
d_{TV}^{\mathcal{B}^{d}}\left( T\left(Q_{nak}\right) ; T\left(\widetilde{P}_{1:k} \right) \right).  
\end{equation}

\noindent
Now, by Scheffe's Lemma (see Lemma 2.1 in \cite{Tsybakov 2009}), we deduce that 
\begin{equation}\label{scheffe}
d_{TV}^{\mathcal{B}^{dk}}\left(Q_{nak};\widetilde{P}_{1:k}\right) = 
\displaystyle\int \left| p(\left. S_{1,k} = t \right| S_{1,n}=na)-p_{\widetilde{S}_{1,k}}(t) \right| dt. 
\end{equation}

\noindent
We easily check the following invariance property: for any $t \in \mathbb{R}^{d}$, 
\begin{equation}\label{invaCond}
p(\left. S_{1,k} = t \right| S_{1,n}=na) = p(\left. \widetilde{S}_{1,k} = t \right| \widetilde{S}_{1,n} = na) = p_{\widetilde{S}_{1,k}}(t) \left( \frac{p_{\widetilde{S}_{k+1,n}}(na-t)}{p_{\widetilde{S}_{1,n}}(na)} \right).
\end{equation}

\noindent
Let $f_{\ell,m}$ be the density of $\widetilde{S}_{\ell,m} := \sum\limits_{j=\ell}^{m} \widetilde{X}_{j}$. Then, by $(\ref{suffi})$, $(\ref{scheffe})$ and $(\ref{invaCond})$, 
\begin{equation}
d_{TV}^{\mathcal{B}^{dk}}\left(Q_{nak};\widetilde{P}_{1:k}\right) = 
\displaystyle\int \left| \frac{f_{k+1,n}(na-t)}{f_{1,n}(na)}-1 \right| f_{1,k}(t) dt. 
\end{equation}

\noindent
Let  $g_{\ell,m}$ be the density of the normalized r.v. associated to $\widetilde{S}_{\ell,m}$. Expressing the $f_{\ell,m}$'s in function of the $g_{\ell,m}$'s, we obtain that 
\begin{equation}\label{renorma}
\frac{f_{k+1,n}(na-t)}{f_{1,n}(na)} = 
\left( \frac{\det \left(\mathrm{Cov}(\widetilde{S}_{1,n})\right)}{\det \left( \mathrm{Cov}(\widetilde{S}_{L_{n}})\right)} \right)^{1/2} 
\frac{g_{k+1,n}(t^{\#})}{g_{1,n}(0)}, 
\end{equation}

\noindent
where $t^{\#} \in \mathbb{R}^{d}$ is $t$ standardized for $u \mapsto f_{k+1,n}(na-u)$, that is,
\begin{equation}
t^{\#} := \mathrm{Cov}\left(\widetilde{S}_{k+1,n}\right)^{-1/2}\left(na-t- \mathbb{E}[\widetilde{S}_{k+1,n}]\right) = (n-k)^{-1/2} \widetilde{B}_{k+1,n} \left[\sum\limits_{i=1}^{k} m_{i}(\theta^{a}_{n})-t \right].
\end{equation}

\noindent
Then, we perform Edgeworth expansions of $g_{k+1,n}$ and $g_{1,n}$ in order to estimate $\frac{f_{k+1,n}(na-t)}{f_{1,n}(na)}$ uniformly in $t$. This is obtained in Lemmas \ref{3.1.DF} and \ref{3.2.DF} below. Finally, we conclude by Lemma \ref{3.3.DF}.

\subsection{Auxiliary Lemmas}\label{auxLem}

Let $t^{\flat} \in \mathbb{R}^{d}$ be $t$ standardized for $f_{1,k}$, that is, 
\begin{equation}
t^{\flat} := \mathrm{Cov}\left(\widetilde{S}_{1,k}\right)^{-1/2}\left(t-\mathbb{E}[\widetilde{S}_{1,k}]\right) = k^{-1/2} \widetilde{B}_{1,k} \left[t-\sum\limits_{i=1}^{k} m_{i}(\theta^{a}_{n})\right].
\end{equation}

\noindent
Therefore, $t^{\flat}$ and $t^{\#}$ are linked by 
\begin{equation}\label{tANDt}
t^{\#} = -\left[ \frac{k}{n-k} \right]^{1/2} \widetilde{B}_{k+1,n} \left(\widetilde{B}_{1,k}\right)^{-1} t^{\flat}. 
\end{equation}

\begin{rem}\label{ttRemark}
By $(\ref{tANDt})$ and since $\frac{k}{n-k}=\frac{k}{n}+O\left(\frac{k^{2}}{n^{2}}\right)$,  we have that 
\begin{equation}
\left\| t^{\#} \right\|^{2} = O\left(\frac{k}{n-k} \right) \left\| t^{\flat}\right\|^{2} = \left[ O\left(\frac{k}{n} \right) + O\left(\frac{k^{2}}{n^{2}} \right) \right] \left\| t^{\flat}\right\|^{2}. 
\end{equation}
\end{rem}

\bigskip

\begin{lem}\label{3.1.DF}
Let $0<\theta_{1}<1$. Then, uniformly in $t$ and $k$, for $k<\theta_{1}n$, 
\begin{equation}\label{avEdge}
\frac{f_{k+1,n}(na-t)}{f_{1,n}(na)} = 
\left[ 1+O\left(\frac{k}{n}\right) \right]^{1/2} 
\exp \left(-\frac{\left\| t^{\#} \right\|^{2}}{2} \right) +
O\left(\frac{\sqrt{k}}{n}\right) \left\| t^{\flat} \right\| + O\left(\frac{1}{n}\right). 
\end{equation}
\end{lem}

\begin{proof}
See Appendix. 
\end{proof}

\begin{lem}\label{3.2.DF}
If $k = o(n)$, and $\left\| t^{\#} \right\| < \theta_{2} < \infty$, then uniformly in $t$, 
\begin{align*} 
\frac{f_{k+1,n}(na-t)}{f_{1,n}(na)} &= 
1+O\left( \frac{k}{n} \right) + O\left( \frac{k}{n} \right) \left\| t^{\flat} \right\|^{2} + O\left( \frac{\sqrt{k}}{n} \right) \left\| t^{\flat} \right\| \\
&+ O\left( \frac{k^{2}}{n^{2}}\right) \left( \left\|t^{\flat}\right\|^{2} + \left\|t^{\flat}\right\|^{4} \right) + O\left( \frac{1}{n} \right) + O\left( \frac{k^{2}}{n^{2}}\right). 
\end{align*}
\end{lem}

\begin{proof}
See Appendix. 
\end{proof}

\begin{lem}\label{3.3.DF}
For $\mu \in \left\{1,2,3,4\right\}$,  
\begin{equation*}
\int \left\| t^{\flat} \right\|^{\mu} f_{1,k}(t)dt = O(1).
\end{equation*}
\end{lem}

\begin{proof}
See Appendix. 
\end{proof}

\subsection{End of proof}

We are now able to prove $(\ref{myDF})$. Setting 
$\kappa(t) := \left| \frac{f_{k+1,n}(na-t)}{f_{1,n}(na)}-1 \right| f_{1,k}(t)$, for fixed $\theta_{2}>0$, 
\begin{equation}
d_{TV}^{\mathcal{B}^{dk}}\left(Q_{nak};\widetilde{P}_{1:k}\right) = 
\int\limits_{\left\|t^{\#}\right\| \leq \theta_{2}}\kappa(t)dt +
\int\limits_{\left\|t^{\#}\right\| > \theta_{2}}\kappa(t)dt.  
\end{equation}

\noindent
Now, Lemmas $\ref{3.2.DF}$ and $\ref{3.3.DF}$ imply that 
\begin{equation}
\int\limits_{\left\|t^{\#}\right\| \leq \theta_{2}}\kappa(t)dt = 
O\left(\frac{k}{n}\right)  
\end{equation}

\noindent
On the other hand, by Lemma \ref{3.1.DF}, 
\begin{equation}\label{3.1.DFconseq}
\frac{f_{k+1,n}(na-t)}{f_{1,n}(na)} = 
\left[ 1+O\left(\frac{k}{n}\right) \right]
\exp \left(-\frac{\left\| t^{\#} \right\|^{2}}{2} \right) +
O\left(\frac{\sqrt{k}}{n}\right) \left\| t^{\flat} \right\| + O\left(\frac{1}{n}\right) +O\left(\frac{k^{2}}{n^{2}}\right), 
\end{equation}

\noindent
which, combined to Lemma $\ref{3.3.DF}$, implies that 
\begin{equation*}
\int\limits_{\left\|t^{\#}\right\|>\theta_{2}}\kappa(t)dt = 
\int\limits_{\left\|t^{\#}\right\|>\theta_{2}} 
\left| \left[1+O\left( \frac{k}{n} \right) \right] 
\exp \left(-\frac{\left\| t^{\#} \right\|^{2}}{2}\right)-1 \right| f_{k}(t)dt + o\left( \frac{k}{n} \right)  
\end{equation*}

\noindent
Now, $\left\|t^{\#}\right\|>\theta_{2}$ implies that there exists a constant $B$, with $0<B<\infty$, such that $\left\|t^{\flat}\right\|> B\left(\frac{n}{k}\right)^{1/2} \theta_{2}$. So, by Markov's inequality and Lemma $\ref{3.3.DF}$, 
\begin{equation}
\int\limits_{\left\|t^{\#}\right\|>\theta_{2}} f_{k}(t)dt
\leq \frac{\int \left\| t^{\flat} \right\|^{4} f_{1,k}(t)dt}{\left[ B\left(\frac{n}{k}\right)^{1/2} \theta_{2} \right]^{4}}
= O\left( \frac{k^{2}}{n^{2}}  \right). 
\end{equation}

\section{Appendix}
\label{sec:5}

\subsection{Proof of Lemma $\ref{lemSteep}$}

\begin{proof}
We adapt the proof of Theorem 9.1. (ii)$^{*}$ in \cite{Barndorff-Nielsen 2014}. It is enough to prove that 
\begin{equation}
\mathrm{int}(C_{X}) \subset \mathrm{dom}((\overline{\kappa}_{n})^{*}) \subset C_{X}. 
\end{equation}

\noindent
We first prove that $\mathrm{dom}((\overline{\kappa}_{n})^{*}) \subset C_{X}$ and we may assume that $C_{X} \neq \mathbb{R}^{d}$. Let $t \notin C_{X}$. Then, there exists a hyperplane $H$ separating $C_{X}$ and $t$. Let $e$ be the unit vector in $\mathbb{R}^{d}$ which is normal to $H$ and such that $C_{X}$ lies in the negative halfspace determined by $H$ and $e$. Since $t \notin C_{X}$, then for all $x \in C_{X}$, $\langle e,t \rangle > 0 \geq \langle e,x \rangle$. We deduce readily that for all $i \geq 1$, since $C_{X_{i}}=C_{X}$ (by $(\mathcal{S}upp)$),  
\begin{equation*}
\exp(-r\langle e,t \rangle) \mathbb{E}[\exp(r\langle e,X_{i} \rangle)] \longrightarrow \infty 
\quad \textrm{as} \quad
r \rightarrow \infty. 
\end{equation*}

\noindent 
Therefore, for all $n \geq 1$, 
\begin{equation*}
\lim\limits_{r \rightarrow \infty} \langle re,t \rangle-\overline{\kappa}_{n}(re) = 
\lim\limits_{r \rightarrow \infty} \frac{1}{n} \left[ \sum\limits_{i=1}^{n} \left( \langle re,t \rangle-\kappa_{i}(re) \right) \right] = \infty.
\end{equation*}

\noindent
So $(\overline{\kappa}_{n})^{*}(t) = \sup\limits_{\theta \in \Theta} \left\{ \langle \theta, t \rangle -\overline{\kappa}_{n}(\theta) \right\} = \infty$, which means that $t \notin \mathrm{dom}((\overline{\kappa}_{n})^{*})$. 

\noindent\\
Next, we prove that $\mathrm{int}(C_{X}) \subset \mathrm{dom}((\overline{\kappa}_{n})^{*})$. Let $t \in \mathrm{int}(C_{X})$. For any $n \geq 1$, set $S_{1,n}:=\sum\limits_{i=1}^{n}X_{i}$. Then, by the independence of the $X_{i}$'s and Jensen's inequality applied to the concave function $u \mapsto u^{1/n}$,   
\begin{equation}\label{JensenInq}
\forall \theta \in \mathbb{R}^{d}, \quad 
\overline{\kappa}_{n}(\theta) \geq \log \mathbb{E}\left[ \exp \langle \theta, S_{1,n}/n \rangle \right]. 
\end{equation}

\noindent
Now, by Markov's inequality, for any $\theta, \tau \in \mathbb{R}^{d}$,
\begin{equation}\label{MarkovInq}
\exp(-\langle \theta,\tau \rangle) \mathbb{E}\left[\exp\langle \theta,S_{1,n}/n \rangle\right] \geq 
\rho_{n}(\tau),   
\end{equation}

\noindent
where $\rho_{n}(\tau) := \inf \left\{ P\left( \langle e, S_{1,n}/n \rangle \geq \langle e, \tau \rangle \right) : \left\|e\right\|=1 \right\}$. Combining $(\ref{JensenInq})$ and $(\ref{MarkovInq})$, we obtain that for any $\theta, \tau \in \mathbb{R}^{d}$, 
\begin{equation}\label{logRho}
\langle \theta,\tau \rangle-\overline{\kappa}_{n}(\theta) \leq 
\langle \theta,\tau \rangle-\log \mathbb{E}\left[ \exp \langle \theta, S_{1,n}/n \rangle \right] \leq -\log \rho_{n}(\tau).
\end{equation} 

\noindent 
Now, since the $X_{i}$'s are independent, $(\mathcal{S}upp)$ implies that $S_{X} \subset \mathrm{supp}(S_{1,n}/n)$. Then, Lemma 9.2. in \cite{Barndorff-Nielsen 2014} asserts that for any $\tau \in \mathrm{int}(C_{S_{1,n}/n})$, we have that $\rho_{n}(\tau)>0$. Since $t \in \mathrm{int}(C_{X}) \subset \mathrm{int}(C_{S_{1,n}/n})$, we deduce from $(\ref{logRho})$ that 
\begin{equation*}
(\overline{\kappa}_{n})^{*}(t) = \sup\limits_{\theta \in \Theta} \left\{ \langle \theta, t \rangle -\overline{\kappa}_{n}
(\theta) \right\} \leq -\log \rho_{n}(t) < \infty, 
\end{equation*}

\noindent 
which means that $t \in \mathrm{dom}((\overline{\kappa}_{n})^{*})$.
\end{proof}

\subsection{Proof of Lemmas $\ref{VdefPos}$-$\ref{BdCf2}$}

\subsubsection{Proof of Lemma $\ref{VdefPos}$}

\begin{proof}
Recall from the Courant-Fischer min-max theorem that for any symmetric matrix $M$,  
\begin{equation}\label{CFmm}
\lambda_{min}(M) = \inf\limits_{x \in \mathbb{R}^{d} \setminus \left\{0\right\}} \frac{x^{t} M x}{x^{t} x}
\qquad \textrm{and} \qquad
\lambda_{max}(M) = \sup\limits_{x \in \mathbb{R}^{d} \setminus \left\{0\right\}} \frac{x^{t} M x}{x^{t} x}.
\end{equation}

\noindent
Fix $n \geq 1$. By $(\ref{CFmm})$, for any $\theta \in K^{a}$, $i \in L_{n}$ and $x \neq 0$,  
\begin{equation*}
\frac{x^{t}C_{i}^{\theta}x}{x^{t}x} \geq 
\lambda_{min}(C_{i}^{\theta}) \geq \lambda_{min}^{K^{a}}. 
\end{equation*}

\noindent
So, by linearity, for any $\theta \in K^{a}$ and $x \neq 0$,
\begin{equation}\label{lineari}
\frac{x^{t} \widetilde{V}_{L_{n}}^{\theta}x}{x^{t}x} = \frac{1}{\left|L_{n}\right|} \sum\limits_{i \in L_{n}} \frac{x^{t}C_{i}^{\theta}x}{x^{t}x} \geq \lambda_{min}^{K^{a}}. 
\end{equation}

\noindent
Then, it follows from $(\ref{CFmm})$ and $(\ref{lineari})$ that
\begin{equation}\label{Lmin}
\lambda_{min}(\widetilde{V}^{\theta_{n}^{a}}_{L_{n}}) \geq
\inf\limits_{\theta \in K_{a}} \lambda_{min} (\widetilde{V}_{L_{n}}^{\theta}) = 
\inf\limits_{\theta \in K_{a}} \left( \inf\limits_{x \in \mathbb{R}^{d} \setminus \left\{0\right\}}  \frac{x^{t} \widetilde{V}_{L_{n}}^{\theta} x}{x^{t}x} \right) \geq \lambda_{min}^{K^{a}}. 
\end{equation}

\noindent
By $\left( \mathcal{C}v \right)$, $\lambda_{min}^{K^{a}}>0$, so  $\widetilde{V}^{\theta_{n}^{a}}_{L_{n}}$ is indeed positive-definite. Now, by $(\ref{CFmm})$, for all $x \in \mathbb{R}^{d}$, 
\begin{equation}
\lambda_{min} \left(\widetilde{B}_{L_{n}} \right) \left\|x\right\| \leq \left\|\widetilde{B}_{L_{n}} x\right\|
\leq \lambda_{max} \left(\widetilde{B}_{L_{n}} \right) \left\|x\right\|. 
\end{equation}

\noindent
By $(\ref{defB})$ and $(\ref{Lmin})$, $\lambda_{max} \left(\widetilde{B}_{L_{n}} \right) = \left[\lambda_{min}\left(\widetilde{V}_{L_{n}} \right)\right]^{-1/2} \leq \left[\lambda_{min}^{K^{a}}\right]^{-1/2}$. Similarly, we have that $\lambda_{min} \left(\widetilde{B}_{L_{n}} \right) \geq \left[\lambda_{max}^{K^{a}}\right]^{-1/2}$. 
\end{proof}

\subsubsection{Proof of Lemma $\ref{AM4lem}$}

\begin{proof}
By Lemma $\ref{VdefPos}$, $\left( \mathcal{C}v \right)$ and $\left( \mathcal{A}m4 \right)$, for all $n \geq 1$ and $i \in L_{n}$, 
\begin{equation*}
\mathbb{E}\left[ \left\| \widetilde{B}_{L_{n}} 
\left( \widetilde{X}_{i}-m_{i}(\theta^{a}_{n})\right) \right\|^{4} \right] 
\leq 
\left[\lambda_{min}^{K^{a}} \right]^{-2} 
~\sup\limits_{i \geq 1} ~\sup\limits_{\theta \in K^{a}}
~\mathbb{E}\left[ \left\| \widetilde{X}_{i}^{\theta}-m_{i}(\theta) \right\|^{4} \right] < \infty. 
\end{equation*}
\end{proof}

\subsubsection{Proof of Lemma $\ref{BdCf1}$}

\begin{proof}
Let $p > \frac{1}{\delta_{K^{a}}}$. Then, for any $n \geq N_{p}$, and $\alpha_{n}-1 \leq m \leq \beta_{n}-p$, 
\begin{equation*}
\int\limits_{\mathbb{R}^{d}} \widetilde{g}_{m,n}^{p}(t)dt =
\int\limits_{ \left\{ \left\| \widetilde{B}_{L_{n}} ~t \right\| \geq R_{K^{a}} \right\} } \widetilde{g}_{m,n}^{p}(t)dt ~ +
\int\limits_{ \left\{ \left\| \widetilde{B}_{L_{n}} ~t \right\| < R_{K^{a}} \right\} } \widetilde{g}_{m,n}^{p}(t) dt. \end{equation*}

\noindent
Then, for $\left\| \widetilde{B}_{L_{n}} ~t \right\| \geq R_{K^{a}}$, we obtain from 
$(\mathcal{C}f1)$ that for all $m+1 \leq i \leq m+p$, 
\begin{equation}\label{xiGros}
\left| \widetilde{\xi}_{i} \left( \widetilde{B}_{L_{n}} ~t \right) \right| \leq \frac{C_{K^{a}}}{\left\| \widetilde{B}_{L_{n}} ~t \right\|^{\delta_{K^{a}}}} \leq 
\frac{C_{K^{a}} 
\left[\lambda_{\max}^{K^{a}}\right]^{\delta_{K^{a}}/2}}{\left\| t \right\|^{\delta_{K^{a}}}}, 
\end{equation}

\noindent 
where the last inequality follows from $(\ref{maxBmin})$. So, by $(\ref{xiGros})$ and the condition on $p$, 
\begin{equation*}
\int\limits_{ \left\{ \left\| \widetilde{B}_{L_{n}} ~t \right\| \geq R_{K^{a}} \right\} } \widetilde{g}_{m,n}^{p}(t)dt \leq 
\left(C_{K^{a}}\right)^{p}
\left[ \lambda_{\max}^{K^{a}} \right]^{(p\delta_{K^{a}})/2}
\int\limits_{\mathbb{R}^{d}} \frac{1}{\left\| t \right\|^{p\delta_{K^{a}}}} dt < \infty. 
\end{equation*}

\noindent
By the change of variables formula, since $\det\left( \widetilde{B}_{L_{n}}\right) \geq  \left[ \lambda_{min}\left( \widetilde{V}_{L_{n}}^{-1/2} \right)\right]^{d} \geq \left[ \lambda_{max}^{K^{a}}\right]^{-d/2}$, 
\begin{equation*}
\int\limits_{ \left\{ \left\| \widetilde{B}_{L_{n}} ~t \right\| < R_{K^{a}} \right\}} \widetilde{g}_{m,n}^{p}(t)dt =
\left[ \det\left( \widetilde{B}_{L_{n}}\right) \right]^{-1}
\int\limits_{ \left\{ \left\| y \right\| < R_{K^{a}} \right\} } \widetilde{g}_{m,n}^{p}(y) dy < \infty.
\end{equation*}
\end{proof}

\subsubsection{Proof of Lemma $\ref{BdCf2}$}

\begin{proof}
Let $p \geq 1$ and $b>0$. For all $n \geq 1$ and $\left\|t\right\|>b$, we deduce from Lemma $\ref{VdefPos}$ that $\left\|\widetilde{B}_{L_{n}} t\right\| \geq \left[\lambda_{max}^{K^{a}}\right]^{-1/2} \left\|t\right\| > \left[\lambda_{max}^{K^{a}}\right]^{-1/2} b$. So, for all $n \geq N_{p}$, $\alpha_{n}-1 \leq m \leq \beta_{n}-p$ and $m+1 \leq i \leq m+p$,
\begin{equation*}
\sup\limits_{\|t\| > b}
~\left| \widetilde{\xi}_{i} \left( \widetilde{B}_{L_{n}} t \right) \right| \leq 
\sup \left\{~\left| \widetilde{\xi}_{i}^{\theta}(u) \right| : i \geq 1 ~;~ \theta \in K ~;~ \left\|u\right\| \geq \left[\lambda_{max}^{K^{a}}\right]^{-1/2} b \right\} < 1, 
\end{equation*}

\noindent
where the last inequality follows from $(\mathcal{C}f2)$ applied with $\beta=\left[\lambda_{max}^{K^{a}}\right]^{-1/2} b$. This implies $(\ref{ConcluBeta})$. 
\end{proof}

\subsection{Proof of Lemmas of section \ref{auxLem}}

\subsubsection{Proof of Lemma \ref{3.1.DF}}

\begin{proof}
The assumptions allow us to perform Edgeworth expansions to obtain that 
\begin{equation}\label{EdgewN}
g_{1,n}(0)=1+O\left(\frac{1}{n}\right)
\end{equation}

\noindent 
and 
\begin{equation}\label{EdgewD}
g_{k+1,n}(t^{\#}) = 
\psi(t^{\#}) \left[ 1 + \frac{1}{(n-k)^{1/2}} \sum\limits_{|\nu|=3} \overline{\chi}_{\nu,L_{n}} H_{3}^{(\nu)}(t^{\#}) \right]
+ O\left(\frac{1}{n-k}\right).
\end{equation}

\noindent
By (6.21) in \cite{Bhattacharya and Rao 1976}, for $\left|\nu \right|=3$, the $\nu$-cumulant of a centered r.v. is equal to its $\nu$-moment. Furthermore, the cumulants are invariant by any translation. So, 
\begin{equation}\label{cumulTau}
\overline{\chi}_{\nu,L_{n}} = 
\frac{1}{n-k} \sum\limits_{k+1 \leq i \leq n} \mathbb{E} [(\widetilde{B}_{L_{n}} \widetilde{X}_{i})^{\nu}]
\leq 
\frac{1}{n-k} \sum\limits_{k+1 \leq i \leq n} 
\mathbb{E} \left[ \left\| \widetilde{B}_{L_{n}} \widetilde{X}_{i} \right\|_{\infty}^{3} \right] = O(1). 
\end{equation}

\noindent
Now, for all $\nu$, $H_{3}^{(\nu)}(t^{\#})=(t^{\#})_{i} P_{i}(t^{\#})$, for some $1 \leq j \leq d$ and a polynomial $P_{i}$. So,  
\begin{equation}\label{Hermite}
\psi(t^{\#})H_{3}^{(\nu)}(t^{\#}) = 
O\left(\left\|t^{\#}\right\|_{\infty}\right) = 
O\left(\sqrt{k/n}\right).\left\|t^{\flat}\right\|. 
\end{equation} 

\noindent
We deduce readily from $(\ref{EdgewN})$, $(\ref{EdgewD})$, $(\ref{cumulTau})$ and $(\ref{Hermite})$ 
that 
\begin{equation}\label{gONg}
\frac{g_{k+1,n}(t^{\#})}{g_{1,n}(0)} = \psi(t^{\#}) +
O\left( \frac{\sqrt{k}}{n} \right) \left\|t^{\flat}\right\| + O\left(\frac{1}{n}\right). 
\end{equation} 

\noindent
On the other hand, 
\begin{equation*}
\Delta_{n} :=
\frac{\det \left( \mathrm{Cov}\left(\widetilde{S}_{1,n}\right) \right)}{\det \left( \mathrm{Cov}\left(\widetilde{S}_{L_{n}}\right)\right)} =
\det \left[ I_{d} + \left( \mathrm{Cov}\left(\widetilde{S}_{k+1,n}\right)\right)^{-1} \mathrm{Cov}\left(\widetilde{S}_{1,k}\right) \right], 
\end{equation*} 

\noindent
where $\left\| \left( \mathrm{Cov}\left(\widetilde{S}_{k+1,n}\right) \right)^{-1} \right\| \leq \left((n-k)\lambda_{min}^{K^{a}} \right)^{-1}$ and $\left\| \mathrm{Cov}\left(\widetilde{S}_{1,k}\right) \right\| \leq k\lambda_{max}^{K^{a}}$. So, 
\begin{equation*}
\left\| \left( \mathrm{Cov}\left(\widetilde{S}_{k+1,n}\right) \right)^{-1} \mathrm{Cov}\left(\widetilde{S}_{1,k}\right) \right\| = O\left( \frac{k}{n-k} \right). 
\end{equation*}  

\noindent
Consequently, performing a Taylor expansion of $\det$ at $I_{d}$, 
\begin{equation}\label{deltaN}
\Delta_{n} = 1+\mathrm{Tr}\left[ \left( \mathrm{Cov}\left(\widetilde{S}_{k+1,n}\right) \right)^{-1} \mathrm{Cov}\left(\widetilde{S}_{1,k}\right) \right] + o\left(\frac{k}{n-k}\right) = 1+O\left(\frac{k}{n}\right).
\end{equation}

\noindent
Then, we deduce $(\ref{avEdge})$ from $(\ref{renorma})$, $(\ref{gONg})$ and $(\ref{deltaN})$. 
\end{proof}

\subsubsection{Proof of Lemma \ref{3.2.DF}}

\begin{proof}
Since $\left\| t^{\#} \right\|$ is bounded, the Taylor-Lagrange inequality and Remark \ref{ttRemark} imply that 
\begin{align*}
\exp \left(-\frac{\left\| t^{\#} \right\|^{2}}{2} \right) &=  
1 -\frac{\left\| t^{\#} \right\|^{2}}{2} + O\left(\left\| t^{\#} \right\|^{4}\right) \\
&= 1 + O\left(\frac{k}{n}\right) \left\| t^{\flat} \right\|^{2} + O\left(\frac{k^{2}}{n^{2}}\right) \left( \left\|t^{\flat}\right\|^{2} + \left\|t^{\flat}\right\|^{4} \right). 
\end{align*}

\noindent
Therefore, in $(\ref{3.1.DFconseq})$, 
\begin{equation*}
\left[ 1 + O\left( \frac{k}{n} \right) \right] 
\exp \left(-\frac{\left\| t^{\#} \right\|^{2}}{2} \right) = 
1 + O\left( \frac{k}{n} \right) + O\left( \frac{k}{n} \right) \left\| t^{\flat} \right\|^{2} + O\left( \frac{k^{2}}{n^{2}} \right) \left( \left\|t^{\flat}\right\|^{2} + 
\left\|t^{\flat}\right\|^{4} \right).
\end{equation*}
\end{proof}

\subsubsection{Proof of Lemma \ref{3.3.DF}}

\begin{proof}
We only need to prove the case $\mu = 4$. Setting $I_{4} := \int \left\| t^{\flat} \right\|^{4}f_{k}(t)dt$, we have that 
\begin{equation*}
I_{4} =
\int\displaystyle
\left\| k^{-1/2} \widetilde{B}_{1,k} \left( \sum\limits_{i=1}^{k} \widetilde{X}_{i}-m_{i}(\theta^{a}_{n})  \right) \right\|^{4}dP 
\leq k^{-2} \left\| \widetilde{B}_{1,k} \right\|^{4} 
\int\displaystyle \left\| \sum\limits_{i=1}^{k} 
\widetilde{X}_{i}-m_{i}(\theta^{a}_{n}) \right\|^{4}dP. 
\end{equation*}

\noindent
For any $n \geq 1$ and $1 \leq i \leq k$, set $Z_{in}:=\widetilde{X}_{i}-m_{i}(\theta^{a}_{n})$. By $(\mathcal{C}v)$, we only need to prove that 
$\mathbb{E} \left[\left\| \sum\limits_{i=1}^{k} 
Z_{in} \right\|^{4}\right]=O(k^{2})$. First, we treat the case $d=1$. Then, 
\begin{equation*}
\left\| \sum\limits_{i=1}^{k} 
Z_{in} \right\|^{4} =
\left(\sum\limits_{i=1}^{k}Z_{in}\right)^{4} = 
\sum\limits_{i=1}^{k}(Z_{in})^{4} + 3\sum\limits_{1\leq i<i' \leq k}(Z_{in})^{2}(Z_{i'n})^{2} + V_{n}, 
\end{equation*}

\noindent
where $\mathbb{E}[V_{n}]=0$. Indeed, for any $n \geq 1$, the 
$(Z_{in})_{1 \leq i \leq k}$ are independent and centered. Now,  
$\left( \mathcal{A}m4 \right)$ implies that 
$\mathbb{E}\left[\sum\limits_{i=1}^{k}(Z_{in})^{4}\right] = O(k)$
 and $\mathbb{E}\left[\sum\limits_{1\leq i<i' \leq k}(Z_{in})^{2}(Z_{i'n})^{2}\right] = O(k^{2})$, which proves the case $d=1$. 

\noindent\\ 
Using that the components of independent r.v.'s are independent, we may conclude the proof by induction, as follows. Assuming that the $(Z_{in})$ are $\mathbb{R}^{d+1}$-valued, 
\begin{equation*}
\left\| \sum\limits_{i=1}^{k}Z_{in} \right\|^{4} = 
\left\{ \left\| \sum\limits_{i=1}^{k}Z_{in} \right\|^{2} \right\}^{2} =
\left\{ \sum\limits_{j=1}^{d} \left[\left(\sum\limits_{i=1}^{k}Z_{in}\right)_{j}\right]^{2}
+ \left[\left(\sum\limits_{i=1}^{k}Z_{in}\right)_{d+1}\right]^{2}
\right\}^{2}.
\end{equation*}

\noindent
Now, we may apply the induction hypothesis after expanding the outermost square. For the cross term, we apply the Cauchy-Schwartz inequality to its expectation. We conclude readily that $I_{4}=O(1)$.  
\end{proof}

\noindent\\
\textbf{Acknowledgements}

\noindent
The author wishes to thank Prof. Fabrice Gamboa for helpful discussions.

\bibliographystyle{natbib}

\end{document}